\documentclass[reqno]{amsart}
\usepackage{amsfonts,amsmath,amsthm,amssymb,latexsym, cite}%
\usepackage{graphicx,verbatim}
\usepackage[colorlinks,plainpages,citecolor=magenta, linkcolor=blue,bookmarksnumbered,backref ]{hyperref}
\theoremstyle{plain}
\newtheorem{thm}{Theorem}

\newtheorem{defn}{Definition}
\newtheorem{cor}{Corollary}
\newtheorem{lem}{Lemma}
\newtheorem{rem}{Remark}
\newtheorem{prop}{Proposition}

\numberwithin{equation}{section}

\newcommand{\dmn}{\mathop{\rm dom}}

\newcommand{\rnk}{\mathop{\rm rank}}
\renewcommand{\Im}{\mathop{\rm Im}}
\newcommand{\supp}{\mathop{\rm supp}}

\renewcommand{\kappa}{\varkappa}

\newcommand{\Real}{\mathbb R}
\newcommand{\Comp}{\mathbb C}
\newcommand{\eps}{\varepsilon}

\newcommand{\cK}{\mathcal{K}}

\newcommand{\cH}{\mathcal{H}}
\newcommand{\cG}{\mathcal{G}}
\newcommand{\cV}{\mathcal{V}}
\newcommand{\cW}{\mathcal{W}}
\newcommand{\cD}{\mathcal{D}}

\newcommand{\cR}{\mathcal{R}}

\newcommand\xe{\left(\tfrac x\eps\right)}
\newcommand\xep{(\eps^{-1}\cdot)}

\newcommand\qdr[1]{{#1}^{[1]}(a)}
\newcommand\qdrz[1]{{#1}^{[1]}(0)}
\newcommand{\ra}{\rangle}
\newcommand{\la}{\langle}

\begin{document}

\title[Quantum graphs: Coulomb-type potentials]
{Quantum graphs: Coulomb-type potentials and exactly solvable models}

\author{Yuriy Golovaty}%
\address{Department of Mechanics and Mathematics,
  Ivan Franko National University of Lviv\\
  1 Universytetska str., 79000 Lviv, Ukraine}
\curraddr{}
\email{yuriy.golovaty@lnu.edu.ua}

\subjclass[2000]{Primary 34L40, 81Q35; Secondary 34E10, 81Q10}

\begin{abstract}
We study the Schr\"{o}dinger operators on a non-compact star graph with the Cou\-lomb-type potentials having singularities at the vertex. The convergence of regularized Hamiltonians $H_\eps$ with cut-off Coulomb potentials coupled with $(\alpha \delta+\beta\delta')$-like ones is investigated.
The $1$D Coulomb potential and the $\delta'$-potential are very sensitive to their regularization method. The conditions of the norm resolvent convergence of $H_\eps$ depending on the regularization are established.   The limit Hamiltonians give the Schr\"{o}dinger operators with the Coulomb-type potentials a ma\-the\-matically precise meaning, ensuring the correct choice of vertex conditions. We also describe all self-adjoint realizations of the formal Coulomb Hamiltonians on the star graph.
\end{abstract}

\keywords{Schr\"{o}dinger operator, Coulomb potential, $\delta'$-potential, quantum graph, vertex coupling condition,  solvable model, point interaction}
\maketitle


\section{Introduction}
In recent decades, the theory of differential operators on metric graphs has been the subject of systematic and extensive study, in particular due to numerous possible applications in solid-state physics and engineering. However, from the point of view of physics, the most exciting application of this theory is quantum graphs.
Quantum dynamics usually exhibits high complexity, especially  propagation in ra\-mi\-ﬁed structures. Quantum graphs provide us with quite effective mathematical models with which we can study the quantum systems in a framework that makes explicit solutions possible. There is extensive  literature on quantum graphs; we refer the reader to \cite{Kuchment2002, BerkolaikoCarlson2006, ExnerKeating2008} and the bibliography therein.

A quantum graph is a metric graph equipped with  Schr\"{o}dinger differential expressions  acting on edges and some conditions at vertices (see, e.g., \cite{BerkolaikoKuchmentBook}).  For quantum graphs, there is a wide variety of conditions coupling the wave functions at vertices, in contrast to quantum systems with point interactions on the line.
This makes the theory of Schr\"{o}dinger operators on graphs much richer.
However, the rich set of possible vertex coupling conditions complicates the construction of solvable models for  specific quantum processes because the problem of choosing physically motivated conditions arises.
The mathematical approach to the construction of such models, in addition to the experimental one, consists in diffe\-rent approximations of Hamil\-to\-nians on graphs. The approximation of a graph by thin tubular domains
is the most natural. The problem then is to analyse the convergence of $3$D Hamiltonians as the tubular networks  shrink to the graph and obtain the limit point interactions at the vertices \cite{MolchanovVainberg2006, MolchanovVainberg2007, Grieser2008, AlbeverioCacciapuotiFinco2007, CacciapuotiExner2007, CacciapuotiFinco2010, Post2012}.
Another way of finding physically motivated point interactions is the approximation of quasi-one-dimensional meso- or nano-scale systems by  Schr\"{o}dinger operators on a graph with potentials singularly perturbed in a neighbourhood of the vertices \cite{ExnerLMP1996, Manko2010, Exner2011, ExnerManko2013, Manko2015}. Then the limit Hamiltonian, obtained in a proper operator topology, is the desired solvable model of the system.
Also interesting is the inverse problem, which consists in constructing approximations of given point interactions at a vertex by some regular Hamiltonians \cite{CheonShigehara1998, CheonExner2004, CheonExnerTurek2010}. It has been proven in \cite{CheonExnerTurek2010} that any singular vertex coupling can be approximated by a family of quantum graphs with simultaneous perturbation of the geometry of the graph and the potentials.

In this paper, we construct solvable models for some non-relativistic quantum processes in ramiﬁed structures.
We study the Schr\"{o}dinger operators on a star graph with the Coulomb-type potentials having  singularities at the vertex. We investigate the convergence of Hamiltonians $H_\eps$ with regularised (cut-off) potentials to give the operators with the Coulomb potentials a mathematically precise meaning and find physically motivated conditions at the vertex. The results obtained in this paper remain valid for general quantum graphs. In this model, there is no interaction between vertices, and only the behavior of the Coulomb-type potential in the vicinity of vertex is significant. Since we are interested in vertex couplings caused by the Coulomb singularity, we study the case of a star-shaped quantum graph.
Our analysis of the Coulomb Hamiltonian can be viewed as a continuation of \cite{GolovatyJMP2019}, where we have found conditions of the norm resolvent convergence of $H_\eps$ on the line and have constructed the solvable models for the one-dimension hydrogen atom.

The rest of the paper is organized as follows.  In Section~\ref{SecResults},  we introduce notation and basic deﬁnitions, and then we state our main results (Theorems~\ref{MainTheorem}-\ref{OprQConvergenceToDirectSum}) on the norm resolvent convergence of $H_\eps$. We also discuss the solvable models constructed in Theorem~\ref{MainTheorem}  and compare them with previous studies dealing with partial cases of the problem. Theorem 4 shows the connection between the resonant decomposition of the graph and the structure of the boundary operator.
In Section~\ref{SecSelfAdj}, we describe all self-adjoint realisations of the formal Schr\"{o}dinger operators on a graph with the Coulomb-type potentials. The proof of the main results is given in Sections~\ref{Sect3} and~\ref{Sect4}.

\section{Main results and discussion}\label{SecResults}

Let $\cG$ be a noncompact  star graph  consisting of $n$ semi-inﬁnite edges $e_1,\dots,e_n$ meeting at a single vertex $a$.  The graph is considered as a planar metric graph with the metrics
coming from the natural embedding of $\cG$ into $\Real^2$.
A function $\phi$ on the graph $\cG$ is the collection $\{\phi_k\}_{k=1}^n$, where $\phi_k\colon e_k\to\Comp$ is a function on the edge $e_k$. Let $\phi(a)$ be the vector $(\phi_1(a),\dots,\phi_n(a))^T$ of the function values at the vertex understood as one-sided limits when the vertex is approached from a
particular edge.  We will also use $\phi'(a)$ to denote the vector $(\phi_1'(a),\dots,\phi_n'(a))^T$, where the one-sided derivative $\phi_k'(a)$ is taken in the direction from the vertex into the edge $e_k$. Here and subsequently, $T$ denotes the transpose. We adhere to the convention that a function $\phi$ belongs to some space $X(\cG)$
if $\phi_k$ belongs to $X(e_k)$ for all $k=1,\dots,n$, i.e., $X(\cG)=\bigoplus_{k=1}^{n}X(e_k)$ and $\|\phi\|_{X(\cG)}=\sum_{k=1}^{n}\|\phi_k\|_{X(e_k)}$.

We define the function $Q$ on $\cG$ as
\begin{equation}\label{CoulombPotential}
 Q_k(x)=\frac{q_k}{|x-a|} \quad\text{for } x\in e_k.
\end{equation}
Here $q_1,\dots,q_n$ are real numbers. We call $Q$ a Coulomb-type potential on the graph (see Figs.~\ref{FigCoulombPtn1} and \ref{FigCoulombPtn2}).
Since $Q$ has a non-integrable singularity at the vertex $a$, we consider the regularized potentials $Q_\eps\colon\cG\to\Real$ given by
\begin{equation}\label{Qeps}
Q_\eps(x)=
  \begin{cases}
   \phantom{\frac{\ln\eps}{\eps}} Q(x), & \text{if \ } |x-a|>\eps,\\
   \frac{\ln\eps}{\eps}\,\kappa\left(\eps^{-1}(x-a)\right), & \text{if \ } |x-a|<\eps,
  \end{cases}
\end{equation}
where $\kappa$ is a real $L^\infty(\cG)$-function  such that $\kappa=0$ for $|x-a|>1$, and $\eps$ is a small positive parameter. We introduce the family of potentials
\begin{equation}\label{Veps}
    W_\eps(x)=Q_\eps(x)+\eps^{-2}V\left(\eps^{-1}(x-a)\right)
    +\eps^{-1}U\left(\eps^{-1}(x-a)\right),
\end{equation}
where $U, V\colon\cG\to\Real$  are bounded measurable functions of compact support. Similarly to the case $n=2$, when $\cG$ can be regarded as the line, we hereafter interpret the families $\eps^{-1}U(\eps^{-1}(x-a))$ and $\eps^{-2}V(\eps^{-1}(x-a))$ as $\delta$- and $\delta'$-like potentials respectively (see Fig.~\ref{FigCoulombPtn3}). Here $\delta$ is Dirac's function. Let
\begin{equation*}
 \cK(\cG)=\left\{\phi\in W_2^2(\cG)\colon \phi \text{ is continuous at }a,\;\textstyle \sum^{n}_{k=1}\phi_k'(a)=0\right\}
\end{equation*}
be the space of $W^2_{2}(\cG)$-functions subject to the Kirchhoff conditions at the vertex~$a$.
We study the  convergence as $\eps\to 0$ of the Schr\"{o}dinger operators
\begin{equation*}
    H_\eps\phi= -\phi''+W_\eps\phi, \qquad \dmn H_\eps=\cK(\cG),
\end{equation*}
where $\phi''$ is the second order derivative of $\phi$ along edges.

\begin{figure}[t]
\centering
  \includegraphics[scale=0.45]{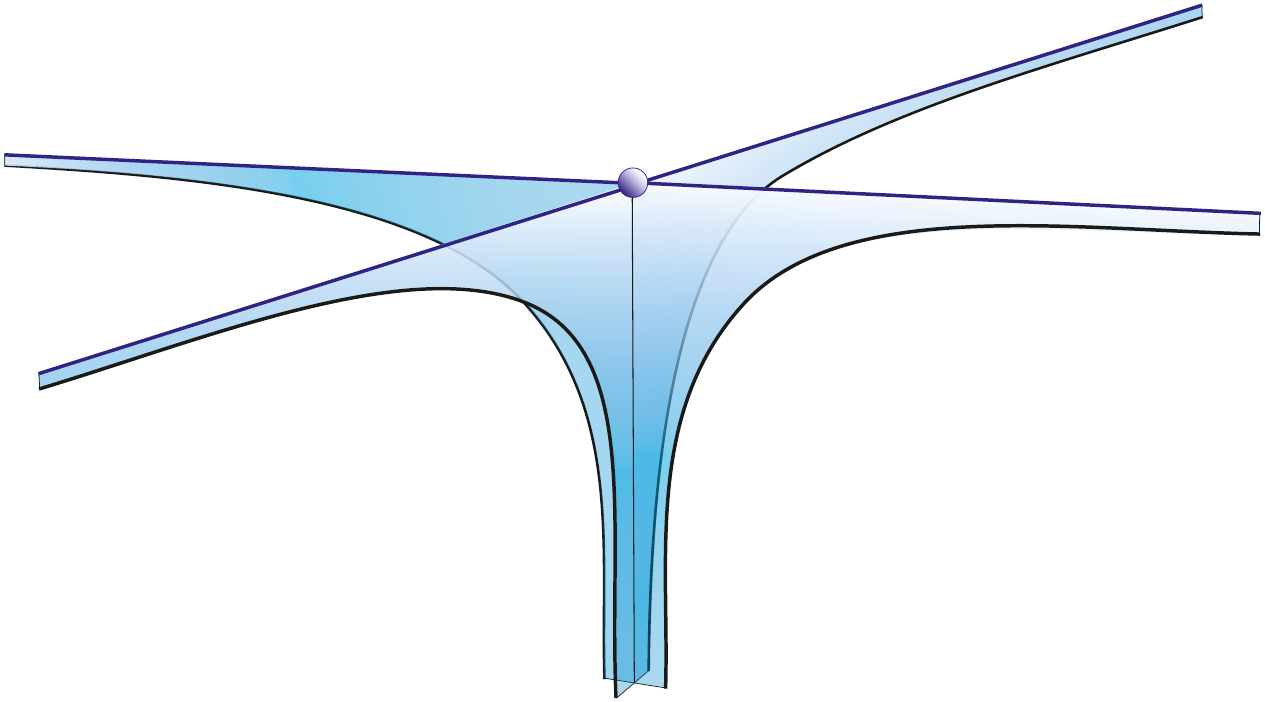}
  \caption{The ``classic'' Coulomb potential with  $q_k=-1$ for all $k$.}\label{FigCoulombPtn1}
\end{figure}

\begin{figure}[b]
\centering
  \includegraphics[scale=0.4]{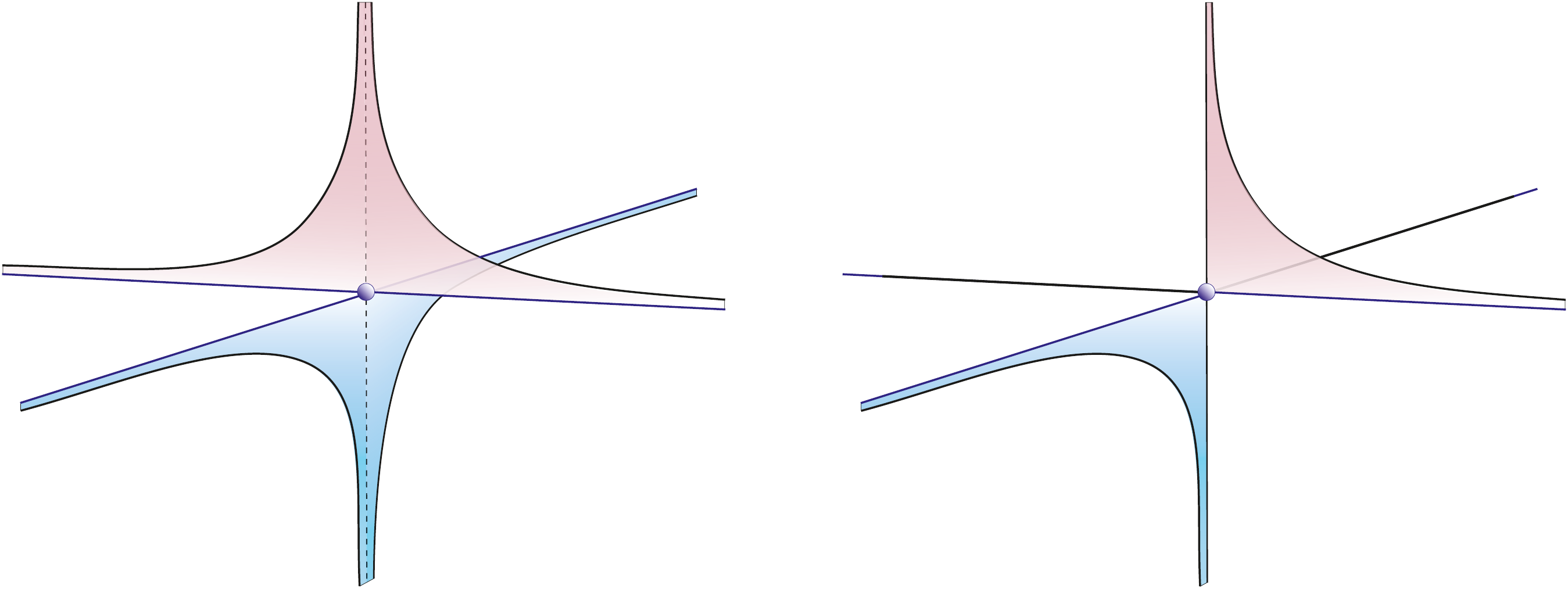}
  \caption{Some exotic Coulomb-type  potentials }\label{FigCoulombPtn2}
\end{figure}

Before stating our main result, we introduce some notation. Let us consider  the Schr\"odinger operator
 \begin{equation*}
 S\psi=-\psi'' + V\psi\quad \text{on } \cG, \qquad \dmn S= \cK(\cG).
\end{equation*}
\begin{defn}\rm
  We say that $S$ possesses  a \emph{zero-energy resonance} (alternatively the potential $V$ is \emph{resonant}) if there exists a non-trivial solution~$\psi$ of the equation $-\psi'' + V\psi=0$ on $\cG$
subject to the Kirchhoff conditions at $a$ that is bounded on the whole graph. We will call $\psi$ the \emph{half-bound state} of $S$.
\end{defn}

This definition generalizes the notion of zero-energy resonance for the Schr\"{o}dinger operators on the line \cite{Klaus1982}.
Denote by $\Psi_V$ the linear space of all half-bound states of $S$. On each edge the half-bound state is constant outside the support of $V$ as a bounded solution of the equation $\psi''=0$. Let $\psi^\infty_k$ be the limit of $\psi$ along the edge $e_k$ as $|x|\to\infty$.  We introduce the map $\ell\colon \Psi_V\to \Comp^n$ defined by $\ell(\psi)=(\psi^\infty_1,\dots,\psi^\infty_n)^T$.

\begin{defn}\rm
 Denote by $\cR_V$ the image of $\Psi_V$ under the map $\ell$  and call it the \emph{resonant space} of the potential $V$. We say that  $V$ is \emph{non-resonant} if $\cR_V$ is trivial.
\end{defn}

If $\psi^\infty_k=0$, then $\psi_k$ is identically zero outside $\supp V$ and hence $\psi_k=0$  on the whole edge $e_k$ as a solution of  the linear differential equation. Therefore $\ell$ is the injection of $\Psi_V$ into $\Comp^n$. So $\dim \cR_V=\dim \Psi_V$. As we will see below, the dimension of $\Psi_V$ cannot exceed $n-1$.
Assume that $\dim \Psi_V=r$ and fix a basis $\psi^{(1)},\dots,\psi^{(r)}$ in $\Psi_V$.
Let $M$ and $N$ be the Hermitian matrices of dimension $r$ with the entries
\begin{equation*}
  m_{ij}=\int_{\cG}U\psi^{(j)}\overline{\psi^{(i)}} \,d\cG, \qquad n_{ij}=\int_{\cG}\kappa\psi^{(j)}\overline{\psi^{(i)}} \,d\cG
\end{equation*}
respectively. Here $d\cG$ is the induced Lebesgue measure on $\cG$.
Also, let $K$ be the diagonal $n\times n$ matrix having $q_k$ from \eqref{CoulombPotential} as its diagonal entries.

We denote by $L$  the $n\times r$ matrix that is formed by the columns
\begin{equation}\label{VectorsLk}
  l_1=\ell(\psi^{(1)}),\; l_2=\ell(\psi^{(2)}),\;\dots\;,\; l_r=\ell(\psi^{(r)}).
\end{equation}
The entry $l_{ij}$ of $L$ is the limit of half-bound state $\psi^{(j)}$ as $x$ goes to infinity along the edge $e_i$.
Let $L^+$ denote the pseudoinverse of $L$. Since $\ell$ is an injection,  $L$ is full column rank and its pseudoinverse can be computed as
\begin{equation}\label{PseudoinverseL}
  L^+=(L^*L)^{-1}L^*,
\end{equation}
where $L^*$ is the Hermitian conjugate of $L$. This particular pseudoinverse constitutes a left inverse. In this case, we have $ L^+L=E$, where $E$ is the $r\times r$ identity matrix.

Given a Coulomb-type potential $Q$,  we  introduce the space
\begin{equation*}
\cW_Q=\left\{\phi\in L^2(\cG)\colon \phi,\, \phi'\in AC_{loc}(\cG),\: -\phi''+Q\phi\in L^2(\cG) \right\},
\end{equation*}
where $AC_{loc}(\cG)$ is the set of absolutely con\-ti\-nuous functions on any compact subset of $\cG$ which does not contain the vertex $a$. In the next section we show that the vector $\phi(a)=(\phi_1(a),\dots,\phi_n(a))^T$ is well defined for every $\phi\in \cW_Q$. However, the vector $\phi'(a)=(\phi_1'(a),\dots,\phi_n'(a))^T$
 may be undefined, because the first derivative of $\phi_k$ may have a logarithmic singularity at $a$. Instead of $\phi'(a)$ we will use the vector $\phi^{[1]}(a)$ consisting of the quasi-derivatives
\begin{equation*}
  \phi^{[1]}_k(a)= \lim_{e_k\ni x\to a}   \big(\phi'_k(x)-q_k \phi_k(a)\ln|x-a|\big).
\end{equation*}

We say the self-adjoint operators $H_\eps$  converge as $\eps\to0$ to the operator $\cH$ in the norm resolvent sense if the resolvents $(H_\eps-\zeta)^{-1}$ converge to $(\cH-\zeta)^{-1}$  in the norm operator topology for all $\zeta\in\Comp\setminus\Real$.
Our main result reads as follows.
\begin{thm}\label{MainTheorem}
Let $\cR_V$ be the resonant space of the potential $V$ present in \eqref{Veps}. Suppose that the following inclusion holds
\begin{equation}\label{ConvergenceCnd}
  \cR_V\subset \ker (NL^+-L^*K).
\end{equation}
Then the operator family $H_\eps$ converges, as $\eps\to0$,  in the norm resolvent sense to the operator $\cH$ defined by $\cH\phi=-\phi''+Q\phi$ on functions $\phi$ from the domain
\begin{equation}\label{DomcH}
  \dmn \mathcal{H}=\left\{\phi\in \cW_Q\colon \phi(a)\in \cR_V, \;\; ML^+ \phi(a)-L^*\qdr\phi=0 \right\}.
\end{equation}
\end{thm}

Also, we can estimate the rate of the resolvent convergence.
\begin{thm}\label{TheoremEst}
Under the assumptions of Theorem~\ref{MainTheorem}, the inequality
\begin{equation*}
  \|(H_\eps-\zeta)^{-1}-(\cH-\zeta)^{-1}\|\leq C(1+|\Im\zeta|^{-1})\,\eps^{1/4}
\end{equation*}
holds for any $\zeta\in \Comp\setminus\Real$ with  $C$  independent of $\eps$ and $\zeta$.
\end{thm}

Let us state several corollaries of the main theorem. Some of them recover known results on the approximations of quantum-graph vertex couplings.

\subsection{$\delta$ couplings on graphs}
The following result itself is so well known that it is even difficult to say by whom it was first published (see, for instance, \cite{ExnerLMP1996}).

\begin{cor}\label{CorDeltaCoupling}
  Assume that $W_\eps(x)=\eps^{-1}U(\eps^{-1}(x-a))$, i.e., $Q$, $V$ and $\kappa$ are equal to zero. Then the operators $H_\eps$ converge in the norm resolvent topology
  to the operator $\cH$ acting as the negative second derivative on functions $\phi$ in $W_2^2(\cG)$  obeying the vertex conditions
  \begin{equation}\label{DeltaConditions}
    \phi_1(a)=\phi_2(a)=\cdots=\phi_n(a), \qquad \sum_{k=1}^n\phi_k'(a)-\phi_1(a)\int_\cG U\,d\cG=0.
  \end{equation}
\end{cor}
\begin{proof}
Our only task is to show how the vertex conditions
\begin{equation}\label{MainCouplConds}
\phi(a)\in \cR_V,\quad ML^+ \phi(a)-L^*\qdr\phi=0
\end{equation}
transform into  \eqref{DeltaConditions}.
Since $Q=0$ and $\kappa=0$, we have $\cW_Q=W_2^2(\cG)$, $\qdr\phi=\phi'(a)$, $K=0$ and $N=0$.
Therefore $NL^+-L^*K=0$, and condition~\eqref{ConvergenceCnd} is trivially fulfilled.

\begin{figure}[b]
\centering
 \includegraphics[scale=0.4]{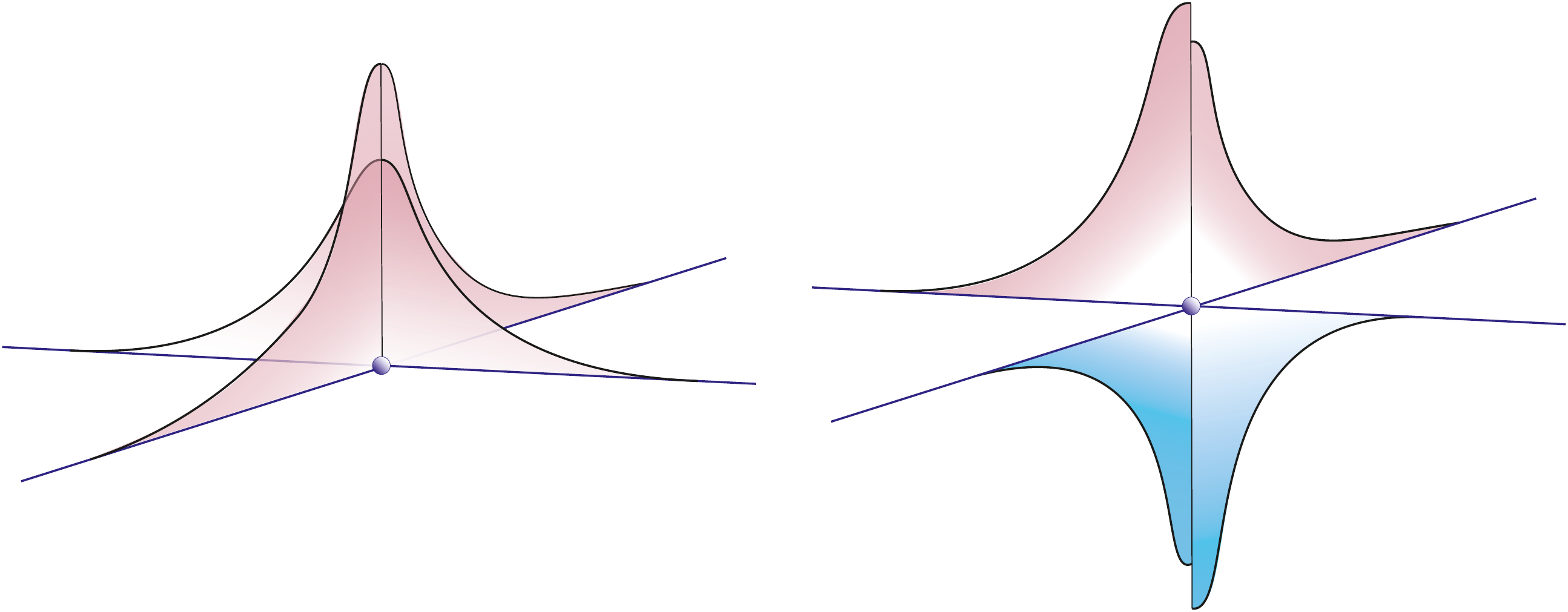}
  \caption{$\delta$- and $\delta'$-like localized potentials }\label{FigCoulombPtn3}
\end{figure}

The potential  $V=0$ is resonant. The corresponding  half-bound states are constant functions on $\cG$. Then $\cR_V$ is a one-dimension space in $\Comp^n$ generated by the vector $\mathbf{1}=(1,1,\dots,1)^T$. Also,  $M$ is the $1\times 1$ matrix with the single entry $m_{11}=\int_\cG U\,d\cG$.
The inclusion $\phi(a)\in \cR_V$ says that  $\phi(a)=c\cdot \mathbf{1}$ for some constant $c$, and therefore $\phi_1(a)=\cdots=\phi_n(a)$.
The matrix $L$ consists of the single column~$\mathbf{1}$. Applying \eqref{PseudoinverseL} we compute  $L^+=\left(\frac1n \;\,\frac1n \;\dots\; \frac1n\right)$. Then we have
\begin{multline*}
  ML^+ \phi(a)-L^*\phi'(a)=
  \begin{pmatrix}
    m_{11}
  \end{pmatrix}
{\textstyle\left(\frac1n \;\,\frac1n \;\dots\; \frac1n\right)}
\begin{pmatrix}
  \phi_1(a)\\
  \phi_2(a)\\
  \ldots\\
  \phi_n(a)
\end{pmatrix}-
  (1\;\, 1\; \dots\; 1)
\begin{pmatrix}
  \phi_1'(a)\\
  \phi_2'(a)\\
  \ldots\\
  \phi_n'(a)
\end{pmatrix}\\
=\frac {m_{11}}n \sum_{k=1}^n\phi_k(a)-\sum_{k=1}^n\phi_k'(a)=m_{11}\phi_1(a)-\sum_{k=1}^n\phi_k'(a)=0,
\end{multline*}
since $\phi$ is continuous at $a$.
\end{proof}

\subsection{$\delta'$ couplings on graphs}
The  Schr\"{o}dinger operators $S_\eps=-\frac{d^2}{dx^2}+\eps^{-2}V(\eps^{-1}x)$ on the line and their  norm resolvent convergence are related to the problem  of $\delta'$-potential and the scattering on localized dipole-type potentials.
In 1986,  \v{S}eba \cite{SebRMP1986} argued that for any potential $V$ of compact support the operators $S_\eps$ converge to the direct sum $S_-\oplus S_+$ of the unperturbed half-line Schr\"{o}dinger operators subject
to the Dirichlet boundary conditions at the origin. One therefore could suggest that no non-trivial interpretation of $\delta'$-potentials is possible. However, such a conclusion  contradicted the numerical analysis of
exactly solvable models with the resonant piecewise constant potentials, performed  by Zolotaryuk et al. \cite{ChristianZolotarIermak03, Zolotaryuk08, Zolotaryuk09}. We have revised \v{S}eba's result in  \cite{GolovatyHrynivJPA2010}. Although  the result remains correct for non-resonant potentials $V$, it has been shown that the resonant potentials produce the limit operator~$S_0$ defined by $S_0\phi=-\phi''$ on functions $\phi$ in~$W_2^2(\Real\setminus\{0\})$ obeying the interface conditions $\psi^-\phi(+0)-\psi^+ \phi(-0)=0$, $\psi^+ \phi'(+0) -\psi^-\phi'(-0)=0$.
Here  $\psi$ is a half-bound state of the operator $-\frac{d^2}{dx^2}+V$ on $\Real$ and $\psi^\pm=\lim_{x\to\pm \infty}\psi(x)$.
In \cite{GolovatyHrynivProcEdinburgh2013}, the result has been extended to potentials from the Faddeev-Marchenko class; see also
\cite{AlbeverioCacciapuotiFinco2007,CacciapuotiExner2007} for the convergence of $S_\eps$ in the case $\int_{\Real} V\,dx\neq 0$ and some applications to quantum waveguides, \cite{Seba1985} for the problem on a half-line, and
\cite{GolovatyMankoDopNAN2009, GolovatyManko2009} for the asymptotics of eigenvalues for  Schr\"{o}dinger operators perturbed by  the $\delta'$-like potentials.

The following statement generalizes the above results to star graphs.

\begin{cor}\label{CorDeltaPrimeCoupling}
Assume that  $W_\eps(x)=\eps^{-2}V(\eps^{-1}(x-a))$ and let $\cR_V$ be the resonant space of $V$.
Then $H_\eps\to \cH$, as $\eps\to 0$, in the norm resolvent sense, where $\cH$
acts as the negative second derivative on functions~$\phi$ from the domain
  \begin{equation*}
  \dmn \cH=\left\{\phi\in W_2^2(\cG)\colon \phi(a)\in \cR_V, \;\; \phi'(a)\in \cR_V^\perp \right\}.
  \end{equation*}
In particular, if the potential $V$ is non-resonant, then
 \begin{equation*}
 \cH=D_1\oplus\cdots\oplus D_n,
\end{equation*}
where $D_k$ is the free  Schr\"{o}dinger operator on $e_k$ subject to the Dirichlet condition $\phi_k(a)=0$ at the vertex.
\end{cor}

\begin{proof}
 As in Corollary~\ref{CorDeltaCoupling},  condition \eqref{ConvergenceCnd} is trivially fulfilled, and hence
the ope\-ra\-tors $H_\eps$ converge. The equality $ML^+ \phi(a)-L^*\qdr\phi=0$ becomes $L^*\phi'(a)=0$, since $M=0$ and $\qdr\phi=\phi'(a)$. Therefore $\la l_1,\phi'(a)\ra_{\Comp^n}=0,\;\dots\;,\:\la l_r,\phi'(a)\ra_{\Comp^n}=0$, where the vectors $l_k$ are given by \eqref{VectorsLk} and  $\langle\cdot,\cdot\rangle_{\Comp^n}$ is the standard Hermitian inner product in $\Comp^n$. This shows that actually $\phi'(a)\in \cR_V^\perp$, because $l_1,\dots,l_r$ form a basis in $\cR_V$. In the case $\cR_V=\{0\}$, the inclusion $\phi(a)\in \cR_V$ says that $\phi_k(a)=0$ for all $k=1,\dots,n$. In addition, the condition $\phi'(a)\in \cR_V^\perp$ trivially holds.
\end{proof}

The vertex conditions described in the corollary are called \emph{scale-invariant}, because they  do not couple the function values $\phi(a)$ and the values of derivatives $\phi'(a)$. Then the operator $\cH$ has no eigenvalues, and its continuous spectrum coincides with $[0,+\infty)$. There are also other properties of the operators with the scale-invariant vertex couplings, such as the factorization
of quantum graph Hamiltonians and the energy independence of vertex scattering matrices  \cite{BerkolaikoKuchmentBook}.

For a $3$-edge star graph these conditions were first obtained by Man'ko \cite{Manko2010} in the context of scattering on the $\delta'$-like potentials. The conditions were written as
\begin{equation*}
  \frac{\phi_1(a)}{\theta_1}=\frac{\phi_2(a)}{\theta_2}=\frac{\phi_3(a)}{\theta_3},\qquad \theta_1\phi_1'(a)+\theta_2\phi_2'(a)+\theta_3\phi_3'(a)=0
\end{equation*}
for the case of simple resonance when  $\cR_V$ is generated by the vector $\theta=(\theta_1,\theta_2,\theta_3)$ and
\begin{equation*}
\eta_1\phi_1(a)+\eta_2\phi_2(a)+\eta_3\phi_3(a)=0,\qquad \frac{\phi_1'(a)}{\eta_1}=\frac{\phi_2'(a)}{\eta_2}=\frac{\phi_3'(a)}{\eta_3}
\end{equation*}
for the case of double resonance when the orthogonal complement of $\cR_V$ in $\Comp^3$ is generated by the vector $\eta=(\eta_1,\eta_2,\eta_3)$. Here we adopt the convention that $\phi_k(a)=0$ if  $\theta_k=0$ and $\phi_k'(a)=0$ if  $\eta_k=0$.

\subsection{$\alpha\delta'+\beta\delta$  couplings on graphs}
The proof of Corollary~\ref{CorDeltaPrimeCoupling}  remains valid for a more general setting. Namely, the negative second derivative operator subject to the  scale-invariant vertex couplings is the limit operator also for the family $H_\eps$ with the potentials
\begin{equation*}
 W_\eps(x)=\eps^{-2}V(\eps^{-1}(x-a))+\eps^{-1}U(\eps^{-1}(x-a))
\end{equation*}
if $M=0$. In the other case, we have the following statement.

\begin{cor}
If the potential  $W_\eps$ has the form $\eps^{-2}V(\eps^{-1}(x-a))+\eps^{-1}U(\eps^{-1}(x-a))$, then the operators $H_\eps$ converge to $\cH$  in the norm resolvent sense, where $\cH$
acts as the negative second derivative on functions~$\phi$ from the domain
\begin{equation*}
  \dmn \cH=\left\{\phi\in W_2^2(\cG)\colon \phi(a)\in \cR_V, \;\;  ML^+ \phi(a)-L^*\phi'(a)=0\right\}.
\end{equation*}
In particular, $\cH=D_1\oplus\cdots\oplus D_n$ if $V$ is non-resonant.
\end{cor}

Exner and Man'ko \cite{ExnerManko2013} have studied the convergence of the Schr\"{o}dinger operators  with the potentials $\frac{\lambda(\eps)}{\eps^2}\,V(\eps^{-1}(x-a))$, where $\lambda(\eps)=1+\lambda \eps+O(\eps^2)$ as $\eps\to 0$. Their results correspond to the case when $U=\lambda V$. In our notation, the vertex couplings obtained in \cite{ExnerManko2013}  have the form
\begin{align}\label{ExMan0}
  &\phi_j(a)-\sum_{i=1}^r l_{i j} \phi_i(a)=0,\quad  j=r+1,\dots,n,\\\label{ExMan1}
  & \phi_i'(a)+\sum_{j=r+1}^n l_{i j} \phi_j'(a)- \sum_{j=1}^r m_{ij} \phi_j(a)=0, \quad i=1,\dots,r.
\end{align}
Let us show that this result is consistent with the above statement. To see this, note that conditions \eqref{ExMan0}, \eqref{ExMan1} are written using  the real matrix $L$ of the form
\begin{equation}\label{MatrixLExMan}
  L=
  \begin{pmatrix}
    \;E\;\\
    L_0
  \end{pmatrix},
\end{equation}
where $E$ is the identity matrix of size $r$ and  $L_0$ is the $(n-r)\times r$ matrix with the real entries $l_{ji}$, $j=r+1,\dots,n$, $i=1,\dots,r$ (see \eqref{VectorsLk}). The fact is that we can choose a real-valued basis in $\Psi_V$ (possibly after rearranging the edges of $\cG$) for which $L$ has the form \eqref{MatrixLExMan}.
To make sure that \eqref{ExMan0}, \eqref{ExMan1} can be written in the form \eqref{MainCouplConds}
we first notice that  the vectors
$\theta_j=(l_{1j},\dots,l_{rj},0,\dots,-1,\dots,0)^T$, $j=r+1,\dots,n$,
with $-1$ in the $j$-th position  are orthogonal to each column $l_k$ of $L$ and form a basis in $\cR_V^\perp$.
Therefore \eqref{ExMan0} is  a collection of equalities $\la\theta_j,\phi(a)\ra_{\Comp^n}=0$, and hence $\phi(a)$ belongs to $\cR_V$. Moreover, for all $\phi(a)\in \cR_V$ we have $\phi(a)=\phi_1(a)l_1+\cdots+\phi_r(a)l_r$.  This equality can be written as $\phi(a)=LP \phi(a)$, where $P$ is the block matrix  $\big(\,E\;0\,\big)$ of size $r\times n$. Since $L^*=L^T$ and $L^+L=E$, we have
\begin{equation*}
ML^+ \phi(a)-L^*\phi'(a)=ML^+LP \phi(a)-L^T\phi'(a)=MP\phi(a)-L^T\phi'(a),
\end{equation*}
and we see at once that $MP\phi(a)-L^T\phi'(a)=0$ is the matrix form of \eqref{ExMan1}.

The list of  standard conditions at a vertex includes the so-called $\delta'$-type conditions
$\phi_1'(a)=\phi_2'(a)=\cdots=\phi_n'(a)$, $\sum_{k=1}^n\phi_k(a)-\alpha\phi_1'(a)=0$,
which is similar to the $\delta$-type one \eqref{DeltaConditions}, with the roles of $\phi(a)$ and $\phi'(a)$ switched. Corollary~\ref{CorDeltaPrimeCoupling} gives us a regularization of these conditions for $\alpha=0$ when the resonant space $\cR_V$ has dimension $n-1$ and the space $\cR_V^\perp$ is generated by the vector $\mathbf{1}$. The case $\alpha\neq 0$ is not covered by our results.

\subsection{The Coulomb couplings on graphs}
The shape $V$ of $\delta'$-like perturbation affects both the convergence condition \eqref{ConvergenceCnd} and the limit operator through the resonances and half-bound states. Indeed, the $\delta'$-like term is the most singular in the potential $W_\eps$ as $\eps\to 0$.  Let us remove this term from $W_\eps$, which will make the Coulomb potential $Q_\eps$ dominant as $\eps\to 0$.

\begin{cor}\label{CorCoulombWithoutV}
  Assume that $W_\eps(x)=Q_\eps(x)+\eps^{-1}U(\eps^{-1}(x-a))$ and the regularization $Q_\eps$ of the Coulomb-type potential $Q$ satisfies the condition
  \begin{equation}\label{QPlusDeltaConvCond}
    \sum_{k=1}^nq_k=\int_\cG \kappa\,d\cG
  \end{equation}
 (cf. \eqref{CoulombPotential} and \eqref{Qeps}).
  Then $H_\eps$ converge in the norm resolvent topology
  to the operator $\cH$ acting as $\cH\phi=-\phi''+Q\phi$ on $\cG$, with the domain of functions $\phi\in\cW_Q$  obeying the vertex conditions
  \begin{equation}\label{QPlusDeltaConditions}
    \phi_1(a)=\phi_2(a)=\cdots=\phi_n(a), \qquad \sum_{k=1}^n\phi_k^{[1]}(a)-\phi_1(a)\int_\cG U\,d\cG=0.
  \end{equation}
Recall that $\phi^{[1]}_k(a)= \lim_{e_k\ni x\to a}   \big(\phi'_k(x)-q_k \phi_k(a)\ln|x-a|\big)$.
 \end{cor}
 \begin{proof}
The potential $V=0$ is resonant and  $\cR_V=\{c\cdot \mathbf{1}, c\in\Comp\}$. Then
 $N$ is a $1\times 1$ matrix with the entry $n_{11}=\int_\cG \kappa\,ds$. In addition, $L= \left(1 \;\,1 \;\dots\; 1\right)^T$ and $L^+=\left(\frac1n \;\,\frac1n \;\dots\; \frac1n\right)$. Hence, we have
 \begin{multline*}
   NL^+-L^*K=
   \begin{pmatrix}
     n_{11}
   \end{pmatrix}
   {\textstyle
   \left(\frac1n \;\,\frac1n \;\dots\; \frac1n\right)}-
     {\textstyle \left(1 \;\,1 \;\dots\; 1\right)}
      \begin{pmatrix}
        q_1&&\\
        &\ddots&\\
        &&q_n
      \end{pmatrix}\\
      =\begin{pmatrix}
        \frac{n_{11}}n-q_1 &\frac{n_{11}}n-q_2&\dots&\frac{n_{11}}n-q_n
      \end{pmatrix}.
 \end{multline*}
 The convergence condition \eqref{ConvergenceCnd} now reads $\mathbf{1}\in \ker (NL^+-L^*K)$. Therefore
 \begin{equation*}
   (NL^+-L^*K)\mathbf{1}=\sum_{k=1}^n\left(\frac{n_{11}}n-q_k\right)=0,
 \end{equation*}
 which gives \eqref{QPlusDeltaConvCond}.
 The limit  vertex  conditions \eqref{QPlusDeltaConditions} can be obtained in the same way as it has been done in the proof of Corollary~\ref{CorDeltaCoupling} for the $\delta$ coupling.
 \end{proof}

We have not been able to prove the convergence of $H_\eps$ when condition \eqref{QPlusDeltaConvCond} does not hold. We suspect that there is, in general, no such convergence. However, even if the operators converge, the limit operator is trivial from the viewpoint of physics. Given $e\in E$, we denote by  $\cD_e$ the operator
\begin{equation*}
\cD_ev=-v''+\frac{q_e}{|x-a|}\,v\quad \text{on } e
\end{equation*}
subject to the Dirichlet condition $v(a) = 0$. Here $q_e=q_k$ for $e=e_k$, cf. \eqref{CoulombPotential}.

\begin{thm}\label{OprQConvergenceToDirectSum}
 Suppose that $W_\eps(x)=Q_\eps(x)+\eps^{-1}U(\eps^{-1}(x-a))$ and the regula\-ri\-zation $Q_\eps$ of a Coulomb-type potential $Q$ does not satisfy condition \eqref{QPlusDeltaConvCond}. If the operators $H_\eps$ converge in the strong resolvent topology, then the limit operator can only be the direct sum $\cD_{e_1}\oplus\cdots\oplus\cD_{e_n}$.
\end{thm}

We do not know publications on the regularization of Coulomb potentials on graphs.  However, there is a large number of works devoted  to the problem on the line.
These studies are related to one-dimensional models of the hydrogen atom; see the pioneering work of Loudon \cite{Loudon1959} and his review \cite{Loudon2016} on the present status of the theory and its relation to relevant experiments. The model has remained a topic of debate and controversy for more than $60$ years.
The main problem here is to provide a  mathematical sense for the formal differential expressions
\begin{equation}\label{pseudoCoulomb}
    -\frac{ d^2\phi}{dx^2}-\frac{\alpha}{|x|}\,\phi=E\phi,
    \qquad
    -\frac{ d^2\phi}{dx^2}-\frac{\alpha}{x}\,\phi=E\phi.
  \end{equation}
The first derivative of  $\phi$  has, in general, a singularity at $x=0$. Therefore the wave function should be subjected to  additional conditions at the origin. For these formal expressions, mathematics gives a large enough set of  coupling conditions that generate  self-adjoint operators in $L^2(\Real)$ \cite{FischerLeschkeMuller1995, deOliveiraVerri2009, BodenstorferDijksmaLanger2000}; the main  issue here is  a  physically motivated choice of such conditions.

Until recently, the even Coulomb potential $-\frac{\alpha}{|x|}$ was considered im\-pe\-netrable; the Hamiltonian of 1D hydrogen atom was assumed to be the direct sum $\cD_-\oplus \cD_+$ of the  half-line Schr\"odinger operators
$\cD_\pm=-\frac{d^2}{dx^2}-\frac{\alpha}{|x|}$ on $\Real_\pm$ with the Dirichlet boundary condition at zero. Many authors have obtained this result \cite{Loudon1959, Gunson1987, HainesRoberts1969, Andrews1976, MehtaPatil1978, Gesztesy1980, Klaus1980, Oseguera_deLlano1993, deOliveiraVerri2012} approximating the even Coulomb potential in different ways, e.g.,
\begin{equation}\label{DifferentRegs}
  -\alpha\min\left\{\dfrac1{|x|},\dfrac1\eps\right\},\quad -\dfrac{\alpha}{|x|+\eps}, \quad -\frac{\alpha }{\sqrt{x^2+\eps^2}},\quad-\frac{\alpha|x|}{x^2+\eps^2}.
\end{equation}
The question of penetrability of the odd Coulomb potential $-\frac{\alpha}{x}$ arose in quark-antiquark systems, namely, in the study of the Dirac oscillator. Moshinsky \cite{Moshinsky1993} was the first who derived the coupling conditions
\begin{equation*}
  \phi(+0)=\phi(-0), \quad \lim_{x\to +0}\big( \phi'(x)+\alpha\phi(0)\ln x\big)=\lim_{x\to -0}\big( \phi'(x)+\alpha \phi(0)\ln (-x)\big).
\end{equation*}
The non-triviality of this solvable model  provoked the dispute between Newton and Moshinsky \cite{Newton1994, Moshinsky1994}; another discussion \cite{FischerLeschkeMuller1995, KurasovJPA1996, FischerLeschkeMuller1997, Kurasov1997} was about the uniqueness of the model.

It is worth noting that the Coulomb potentials and the pseudopotential $\delta'$  are very sensitive to their regularization method \cite{HainesRoberts1969, FischerLeschkeMuller1995, GolovatyJMP2019, GolovatyHrynivJPA2010, GolovatyMFAT2012, GolovatyHrynivProcEdinburgh2013}.
In particular, the one-dimensional model of the hydrogen atom shows a critical dependence on the behaviour of cut-off Coulomb potentials at small distances.
In \cite{GolovatyJMP2019}, we constructed non-trivial solvable models for 1D Schr\"{o}dinger differential expressions with the potentials
\begin{equation*}
Q(x)=
  \begin{cases}
    \frac{q_1}{x}, & \text{if \ } x<0,\\
    \frac{q_2}{x}, & \text{if \ } x>0.
  \end{cases}
\end{equation*}
The key point of the study was the choice of  regularizations that converge in the space of distributions. The function $Q$ is not a distribution on the line because it is non-integrable at the origin, except for $q_1=q_2=0$.
However, its primitive (in the classical sense)
\begin{equation*}
F(x)=
  \begin{cases}
    q_1\ln(-x), & \text{if \ } x<0,\\
    q_2\ln x, & \text{if \ } x>0
  \end{cases}
\end{equation*}
is already a distribution. Let $\hat{Q}$ be the distributional derivative of $F$. Then there exist infinitely many distributions
\begin{equation*}
 f=\hat{Q}+\sum_{k=0}^m c_k\delta^{(k)}(x), \quad m\in \mathbb{N}\cup\{0\}, \quad c_k\in\Comp
\end{equation*}
that coincide with $Q$ outside the origin, that is to say $f(\phi)=\int_\Real Q\phi\,dx$ for all test functions $\phi\in C_0^\infty\left(\Real\setminus\{0\}\right)$. Restricting ourselves to the distributions of the form $\hat{Q}+c_0\delta(x)+c_1\delta'(x)$, we studied the  norm resolvent convergence of the Schr\"{o}dinger operators with the regularized potentials $Q_\eps+\eps^{-1}U(\eps^{-1}\,\cdot\,)+\eps^{-2}V(\eps^{-1}\,\cdot\,)$, where
\begin{equation*}
Q_\eps(x)=
  \begin{cases}
   \phantom{\frac{\ln\eps}{\eps}} Q(x), & \text{if \ } |x|>\eps,\\
   \frac{\ln\eps}{\eps}\,\kappa\left(\frac{x}{\eps}\right), & \text{if \ } |x|<\eps.
  \end{cases}
\end{equation*}
It is easy to check that $Q_\eps$ converge in $\cD'(\Real)$ if and only if
\begin{equation}\label{OneDimCond}
q_2-q_1=\int_{\Real}\kappa\,dx.
\end{equation}
A partial result of \cite{GolovatyJMP2019}, which is analogous to Corollary~\ref{CorCoulombWithoutV}, is as follows. If condition \eqref{OneDimCond} holds, then the operators $H_\eps= -\frac{d^2}{dx^2}+Q_\eps(x)
    +\frac{1}{\eps}\,U\left(\frac{x}{\eps}\right)$
converge in the norm resolvent sense to the operator $H_0=-\frac{d^2}{dx^2}+Q$ with the point interactions
\begin{equation}\label{CoulombPointInteractions}
\begin{aligned}
  \phi(+0)=\phi(-0), \quad \lim_{x\to +0}\big(& \phi'(x)-q_2\phi(0)\ln x\big)\\
  &-\lim_{x\to -0}\big( \phi'(x)-q_1\phi(0)\ln (-x)\big)= \phi(0)\int_\Real U\,dx.
\end{aligned}
\end{equation}
Otherwise, the limit operator is the direct sum of the half-line operators $-\frac{d^2}{dx^2}+\frac{q_1}{x}$ and $-\frac{d^2}{dx^2}+\frac{q_2}{x}$ subject to the Dirichlet boundary conditions at the origin.

Kurasov \cite{KurasovJPA1996, Kurasov1997} was the first to draw attention to the connection between the non-trivial solvable models for the Coulomb potentials and the treatment of these potentials as distributions (see also \cite{Gunson1987, Lunardi2019}). He obtained point interactions \eqref{CoulombPointInteractions} for the case $q_1=q_2=-\alpha$ by interpreting the formal differential expression $-\frac{ d^2}{dx^2}-\frac{\alpha}{x}$ as the map
\begin{equation*}
  H=P.V.\left(-\frac{ d^2}{dx^2}-\frac{\alpha}{x}\right),
\end{equation*}
from some Hilbert space to the space of distributions. Here $P.V.$ stands for the Cauchy principal value.

Note that each family of potentials \eqref{DifferentRegs} can be written as $Q^{(0)}_\eps+\eps^{-1}U(\eps^{-1}\cdot)$ for a suitable function $U$, where
\begin{equation*}
  Q^{(0)}_\eps(x)=
  \begin{cases}
   -\frac{\alpha}{|x|}, & \text{if \ }  |x|>\eps,\\
   \hskip14pt 0, & \text{if \ } |x|<\eps.
  \end{cases}
\end{equation*}
These regularizations correspond to the case $q_1=-\alpha$, $q_2=\alpha$ and $\kappa=0$. Hence all of them lead to  trivial point interactions since condition \eqref{OneDimCond} does not hold (compare with Theorem~\ref{OprQConvergenceToDirectSum}).

Thus, there is no unique model of the hydrogen atom described by  the pseudo-Hamiltonians in \eqref{pseudoCoulomb}. There are many solvable models that correspond to the formal Coulomb Hamiltonians. These models describe quantum systems with quite different properties.

 \subsection{Resonant decomposition of the graph and decoupling conditions}
In the limit $\eps\to 0$, the quantum system with the Hamiltonian $H_\eps$ can split into disconnected subsystems, which is affected by the structure of space $\Psi_V$ (or the resonant space $\cR_V$).
Let $E'$ be a subset of $E=\{e_k\}_{k=1}^n$ and let $\cG(E')$ be the subgraph of $\cG$ with the edges from $E'$. We introduce the space
\begin{equation*}
  X(E')=\left\{\psi\in \Psi_V\colon \supp \psi\subset \cG(E')\right\}
\end{equation*}
of all half-bound states that vanish outside $\cG(E')$. This space is non-trivial only if $E'$ contains at least two edges. Also, we call an edge  of $\cG$  \textit{non-resonant} if all half-bound states vanish on it.

\begin{thm}\label{CorNonResonance}
Let $E_0$ be the set of non-resonant edges of $\cG$ for given potential $V$.
Assume that $E_0, E_1,\dots, E_m$
is a partition of $E$ and the space $\Psi_V$ of half-bound states admits the decomposition
\begin{equation*}
    \Psi_V=X(E_1)\oplus\dots\oplus X(E_m).
  \end{equation*}
If $H_\eps$ converge, then  the limit operator $\cH$  can be represented as
$$
\cH=\Big(\bigoplus_{e\in E_0}\cD_e\Big)\oplus\cH(\cG_1)\oplus\cdots\oplus\cH(\cG_m),
$$
where the operator $\cH(\cG_k)$  acts on the subgraph $\cG_k=\cG(E_k)$ only (we will specify $\cH(\cG_k)$ in the proof).
If the potential $V$ is non-resonant,  then  $\cH=\cD_{e_1}\oplus\cdots\oplus\cD_{e_n}$.
\end{thm}

\begin{figure}[b]
\centering
  \includegraphics[scale=0.55]{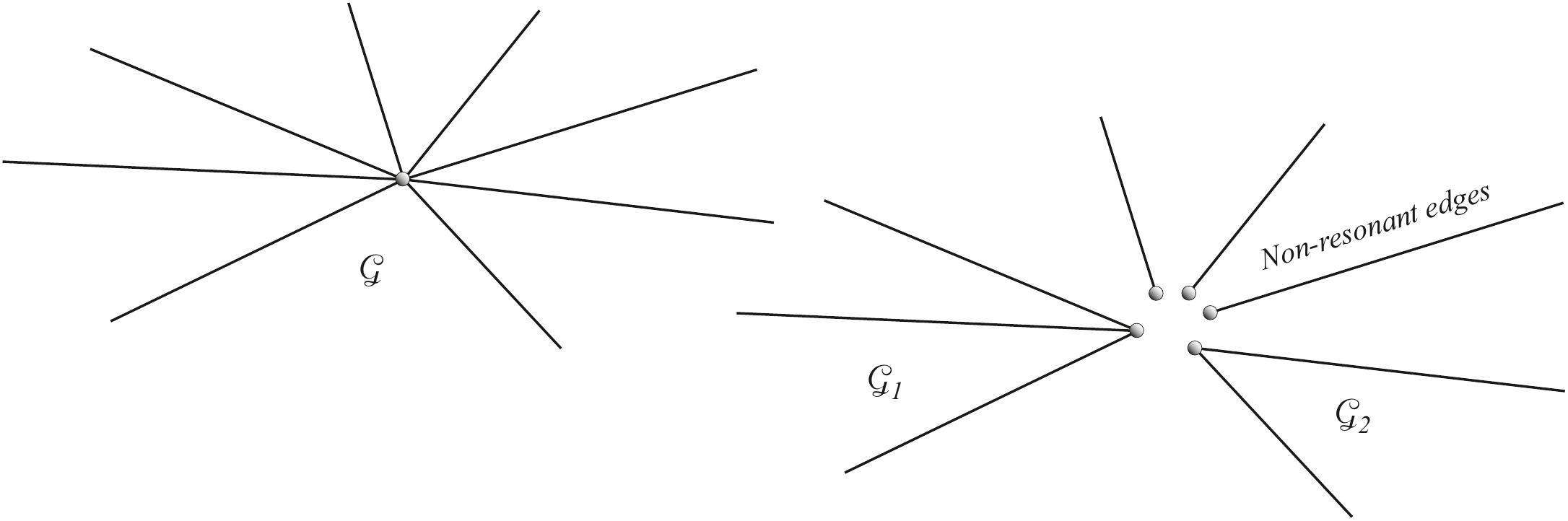}
  \caption{The resonant decomposition of the graph $\cG$}\label{ResonantDecomposition}
\end{figure}

\begin{proof}
We give the proof only for the case $\Psi_V=X(E_1)\oplus X(E_2)$; similar arguments apply to the general case. Let us rearrange the edges so that $E_1=\{e_1,\dots,e_{p_1}\}$, $E_2=\{e_{p_1+1},\dots,e_{p_2}\}$ for some $2\leq p_1<p_2 \leq n$ (see Fig.~\ref{ResonantDecomposition}). Then $E_0=\{e_{p_2+1},\dots,e_n\}$. Set $\dim X(E_k)=r_k$.
By the definition of  $X(E_k)$ the  matrices $M$ and $N$ have the block structure
\begin{equation*}
  M=
  \begin{pmatrix}
     M_1 &0  \\
      0  &  M_2
  \end{pmatrix}
,\qquad
  N=
  \begin{pmatrix}
     N_1 & 0 \\
      0  &  N_2
  \end{pmatrix},
\end{equation*}
where $M_k$ and $N_k$ are $r_k\times r_k$ blocks. Let $R_k$ be the image of $X(E_k)$ under the map $\ell$.
Then $\cR_V=R_1\oplus R_2$ and the matrices
\begin{equation*}
  L=
  \begin{pmatrix}
    L_1 & 0\\
    0 & L_2\\
    0 &0
  \end{pmatrix},\quad
  L^*=\begin{pmatrix}
    L_1^* & 0& 0\\
    0 & L_2^*& 0
  \end{pmatrix},\quad
  L^+=\begin{pmatrix}
    L_1^+  & 0& 0\\
    0 & L_2^+& 0
    \end{pmatrix}
\end{equation*}
also have the block structure. Here $L_1$ is a block of size $r_1\times p_1$ and $L_2$ is a block of size $r_2\times (p_2-p_1)$.
In addition, if $h\in \cR_V$, then $h_{p_2+1}=\dots=h_n=0$.
Therefore  conditions \eqref{MainCouplConds} can be written as
\begin{gather*}
  u(a)\in R_1,\quad M_1L_1^+ u(a)-L_1^*\qdr u=0, \\
   v(a)\in R_2,\quad M_2L_2^+ v(a)-L_2^*\qdr v=0, \\
   \phi_{p_2+1}(a)=0, \; \phi_{p_2+2}(a)=0,\;\dots,\;\phi_{n}(a)=0,
\end{gather*}
where $u$ and $v$ are the restrictions of $\phi$ to the subgraphs $\cG_1$ and $\cG_2$ respectively.
These decoupling conditions describe the direct sum
\begin{equation*}
  \cH=\Big(\bigoplus_{e\in E_0}\cD_e\Big)\oplus\cH(\cG_1)\oplus\cH(\cG_2).
\end{equation*}

If the space of half-bound states is trivial, then all edges of $\cG$ are non-resonant. Hence $E=E_0$ and so $\cH=\cD_{e_1}\oplus\cdots\oplus\cD_{e_n}$.
\end{proof}

\section{Self-adjointness of the limit operator}\label{SecSelfAdj}

Given $f\in L^2(\cG)$ and $\zeta\in \mathbb{C}$, we consider the differential equation
\begin{equation}\label{EqnOnlyOnGraph}
  -z''+(Q-\zeta)z=f \quad \text{on } \cG.
\end{equation}
We impose no conditions on $z$ at the vertex $a$, so  \eqref{EqnOnlyOnGraph} is the collection of equations
\begin{equation*}
  -z_k''+\left(\frac{q_k}{|x-a|} -\zeta\right)z_k=f_k\quad \text{on }e_k,\qquad  k=1,\dots,n.
\end{equation*}

\begin{prop}\label{PropYasymp}
Let $z$ be a $L^2(\cG)$-solution of \eqref{EqnOnlyOnGraph}. For each $k=1,\dots,n$ there exists a finite value $z_k(a)$ and the first derivative of $z_k$ has asymptotics
\begin{equation}\label{LogAsympOfDerv}
  z'_k(x)=q_k z_k(a)\ln|x-a|+b_k+o(1) \quad \text{as } e_k\ni x\to a,
\end{equation}
where $b_k$ is a constant \cite{Kurasov1997}.
\end{prop}

The following theorem is a generalization of the results from \cite{FischerLeschkeMuller1995, deOliveiraVerri2009, BodenstorferDijksmaLanger2000} for the case of quantum graphs.

\begin{thm}\label{OprQisSelfAdjoint}
  Let $Q\colon \cG\to\Real$ be the Coulomb-type potential given by \eqref{CoulombPotential}. Consider the operator $H$ acting as $H\phi=-\phi''+Q\phi$ on $\cG$ with the domain
 \begin{equation*}
   \dmn H =\big\{\phi\in \cW_Q\colon A\phi(a)+B\qdr{\phi}=0\big\},
 \end{equation*}
where $A$ and $B$ are  $n\times n$-matrices. The operator $H$ is self-adjoint in $L_2(\cG)$ if and only if
the $n\times 2n$ matrix $(A\,B)$ has the maximal rank and the matrix $AB^*$ is self-adjoint.
\end{thm}
\begin{proof}
In the case $Q=0$, this result was obtained by Kostrykin and Schrader \cite{KostrykinSchrader1999}.  They have described all possible self-adjoint vertex conditions
\begin{equation}\label{ABconditions}
  A\phi(a)+B\phi'(a)=0
\end{equation}
for the second derivative operator on $\cG$ in terms of the complex symplectic geo\-met\-ry. We recall briefly  their proof.   Let us consider the skew-Hermitian quadratic form $\Omega(\phi, \psi)=(\phi'',\psi)_{L_2(\cG)}-(\phi, \psi'')_{L_2(\cG)}$ for $\phi, \psi \in W_2^2(\cG)$. Let $\langle\;\rangle\colon W^2_2(\cG)\to \Comp^{2n}$ be  the  linear map $\langle\phi\rangle=(\phi(a), \phi'(a))^T$.
Integrating by parts we find
\begin{equation*}
  \Omega(\phi, \psi)=\sum_{i=1}^n(\phi_{i}(a) \overline{\psi_i'}(a)-\phi_i'(a) \overline{\psi_i}(a))=:\omega(\langle\phi\rangle,\langle\psi\rangle).
\end{equation*}
Let the linear subspace $\cV_{AB}$ of $\Comp^{2n}$ be given as the set of all vectors $\langle\phi\rangle$ satisfying \eqref{ABconditions}. To find all self-adjoint vertex conditions of type \eqref{ABconditions} it suffices to find all maximal isotropic subspaces $\cV_{AB}$ with respect to the Hermitian symplectic form~$\omega$.
But $\cV_{AB}$ is maximal isotropic if and only if  $\rnk(A\,B)=n$  and  $AB^*=BA^*$.

Suppose now that $Q$ is different from zero.   In view of Proposition~\ref{PropYasymp}, the vectors $\phi(a)$ and $\phi ^{[1]}(a)$ are well defined for any $\phi\in \cW_Q$. Moreover the form $\Omega$ can be extended to the space $\cW_Q$:
\begin{equation*}
  \Omega(\phi, \psi) =\sum_{k=1}^n\big(\phi_{i}(a) \overline{\psi_k}^{[1]}(a)-\phi_k^{[1]}(a) \overline{\psi_k}(a)\big),\quad \phi, \psi \in\cW_Q.
\end{equation*}
Indeed, let $\cG_\eps=\cG\cap\{x\colon |x-a|>\eps\}$ and let $a^\eps_k$ be the intersection point of the edge $e_k$ with the sphere $\{x\colon |x-a|=\eps\}$.
Then
\begin{multline*}
  \Omega(\phi, \psi)
  =\lim_{\eps\to 0}\big((\phi'',\psi)_{L_2(\cG_\eps)}-(\phi, \psi'')_{L_2(\cG_\eps)}\big)\\
\begin{aligned}
   &=\lim_{\eps\to 0}\sum_{k=1}^n\big(\phi_k(a^\eps_k) \overline{\psi_k'}(a^\eps_k)-\phi_k'(a^\eps_k) \overline{\psi_k}(a^\eps_k)\big)\\
  &=\sum_{k=1}^n\Big(\phi_k(a) \lim_{\eps\to 0}\big(\overline{\psi_k'}(a^\eps_k)-q_k \overline{\psi_k}(a)\ln|a^\eps_k-a|\big)\\
  &\hskip22pt-\lim_{\eps\to 0}\big(\phi_k'(a^\eps_k)-q_k \phi_k(a)\ln|a^\eps_k-a|\big)\; \overline{\psi_k}(a)\Big)
\end{aligned}
\\
  = \sum_{k=1}^n\big(\phi_{i}(a) \overline{\psi_k}^{[1]}(a)-\phi_k^{[1]}(a) \overline{\psi_k}(a)\big),
\end{multline*}
since $a^\eps_k\to a$ as $\eps\to 0$.  Hence $\Omega(\phi, \psi)=\omega(\langle\phi\rangle_Q,\langle\psi\rangle_Q)$,
where $\langle\;\rangle_Q\colon \cW_Q\to \Comp^{2n}$ is  the  linear map given as $\langle\phi\rangle_Q=(\phi(a), \phi^{[1]}(a))^T$.
Suppose that $\cV_{AB}^Q$ is the linear subspace in $\Comp^{2n}$ consisting of all vectors $\langle\phi\rangle_Q$ which satisfy the condition  $A\phi(a)+B\qdr{\phi}=0$. Therefore the operator $H$ is self-adjoint iff $\cV_{AB}^Q$ is maximal isotropic with respect to $\omega$ and iff $\rnk(A\,B)=n$  and  $AB^*$ is self-adjoint.
\end{proof}

\begin{lem}\label{SelfadjointnessH}
The operator $\cH$ from Theorem~\ref{MainTheorem} is self-adjoint in $L_2(\cG)$.
\end{lem}
\begin{proof}
Recall that the vectors $l_1,\dots,l_r$ given by \eqref{VectorsLk} are the columns of  $L$ that form a basis in $\cR_V$. Let $l_{r+1},\dots,l_n$ be the basis of $\cR_V^\bot$ and let the $n\times (n-r)$ matrix $R$ be formed from these vectors as columns. Now the condition $\phi(a)\in \cR_V$ in \eqref{DomcH} can be written as $R^*\phi(a)=0$, since $\cR_V=\ker R^*$. Then
\begin{equation*}
   \dmn \cH =\big\{\phi\in \cW_Q\colon A\phi(a)+B\qdr{\phi}=0\big\},
 \end{equation*}
where the matrix $(A\,B)$ has the $2\times 2$ block-matrix form

\begin{center}
\includegraphics[scale=0.3]{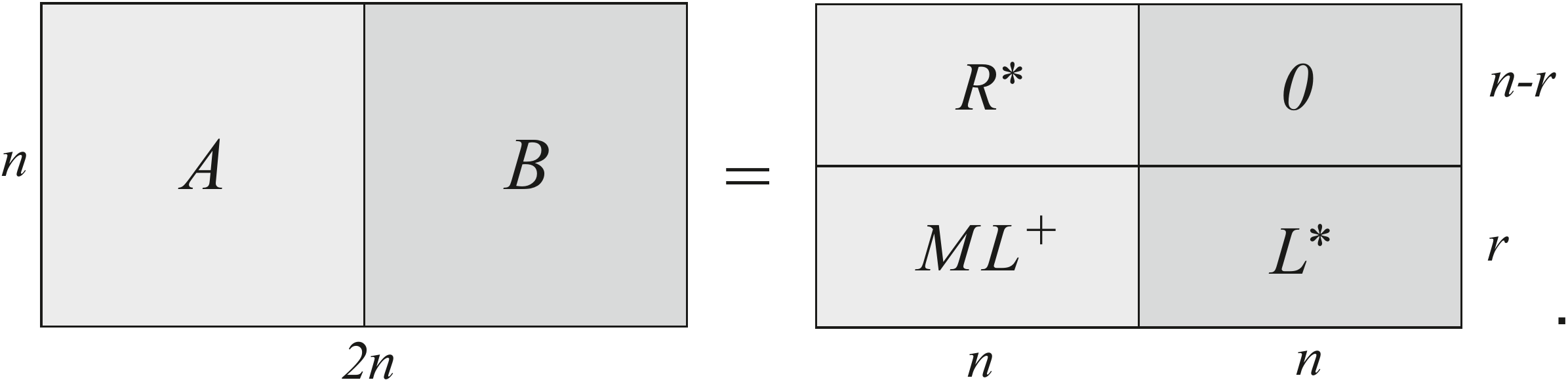}
\end{center}

\noindent
The matrix has the maximal rank because the rows of its blocks $R^*$ and $L^*$ form a basis of $\Comp^{n}$. Next, we have
\begin{center}
  \includegraphics[scale=0.3]{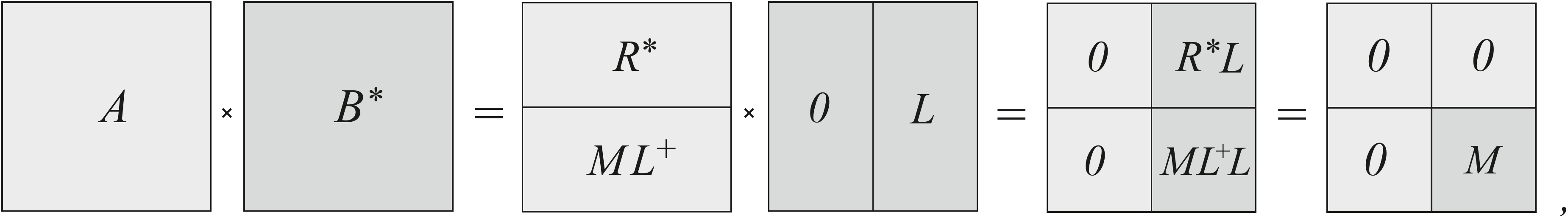}
\end{center}
since $ML^+L=M$ and $R^*L=0$, by construction. Hence the matrix $AB^*$ is self-adjoint, because of $M^*=M$. Therefore, in view of Theorem~\ref{OprQisSelfAdjoint}, the operator $\cH$ is self-adjoint.
\end{proof}

\begin{lem}
The operator $\cH$ and condition \eqref{ConvergenceCnd} in Theorem~\ref{MainTheorem} do not depend on the choice of  basis in  $\Psi_V$.
\end{lem}
\begin{proof}
Suppose the basis $\{\psi^{(k)}\}_{k=1}^r$ in $\Psi_V$  is mapped onto another one $\{\hat{\psi}^{(k)}\}_{k=1}^r$ by a nondegenerate matrix $X$. In the new basis, we have $\hat{L}=LX$, $\hat{M}=X^*MX$  and $\hat{N}=X^*NX$.
Since $X^+=X^{-1}$, we derive
\begin{align}
&\begin{aligned}\nonumber
   \hat{N}\hat{L}^+-\hat{L}^*K=X^*&NX(LX)^+-(LX)^*K\\
   &=X^*NXX^{-1}L^+ -X^* L^*K=X^*\big(NL^+ -L^*K\big),
\end{aligned}
 \\
&\begin{aligned}\label{MLX}
   \hat{M}&\hat{L}^+ \phi(a)-\hat{L}^*\qdr\phi=X^*MX(LX)^+ \phi(a)-(LX)^*\qdr\phi\\
   &=X^*MXX^{-1}L^+ \phi(a)-X^* L^*\qdr\phi=X^*\big(ML^+ \phi(a)-L^*\qdr\phi\big).
\end{aligned}
\end{align}
Hence,  $\ker(\hat{N}\hat{L}^+-\hat{L}^*K) =\ker(NL^+-L^*K)$ and
the left-hand side of \eqref{MLX} is zero if and only if $ML^+ \phi(a)-L^*\qdr\phi=0$. Note that
the condition $\phi(a)\in \cR_V$ in \eqref{DomcH} is also invariant under such a transformation.
\end{proof}

\section{Proof of Theorems~\ref{MainTheorem} and \ref{TheoremEst}}\label{Sect3}
\subsection{Formal construction of  the limit operator}
From now on, we suppose that the vertex $a$ is situated at the origin, and  each edge $e_k$ is a ray on the  plane with beginning at $a=0$. We also assume that the supports of the potentials $V$ and $U$ lie in the unit ball $\{x\colon |x|\leq 1\}$, which involves no loss of generality. Let $\vec{\imath}_k$ be the unit vector along the ray $e_k$. Then
\begin{equation}\label{Parametrization}
 x(\tau)=\tau \vec{\imath}_k, \qquad \tau\in[0,+\infty),
\end{equation}
is a parametrization of $e_k$. It will cause no confusion if we use $\phi_k(x)$ as well as $\phi_k(\tau)$ to designate the value of $\phi$ at a point $x\in e_k$.

Set $\Gamma_\eps=\cG\cap\{x\colon |x|\leq \eps\}$ and $\cG_\eps=\cG\setminus\Gamma_\eps$. Let $\Gamma$ be the dilatation of the localized graph $\Gamma_\eps$ by the factor $1/\eps$ (see Fig.~\ref{GammaEps}). We will denote by $\partial\Gamma$ the boundary of the star graph, which consists of all vertices $\{a_1,\dots,a_n\}$ of $\Gamma$ except the central one. Let $\cK(\Gamma)$
be the space of $W^2_{2}$-functions on $\Gamma$ subject to the Kirchhoff conditions at the vertex~$a=0$.

\begin{lem}\label{LemNminus1}
  The dimension of $\Psi_V$ cannot exceed $n-1$.
\end{lem}
\begin{proof}
A half-bound state $\psi$ is a non-trivial solution of the equation
$-\psi'' + V\psi=0$ on $\cG$ subject to the Kirchhoff conditions at $a=0$.
Since $\psi$ is bounded on the graph, it is  constant  on each edge outside the ball $\{x\colon |x|\leq 1\}$. Then the restriction of $\psi$ to the graph $\Gamma$ solves the boundary value problem
\begin{equation}\label{NeumannPrb}
  -\psi'' + V\psi=0\quad \text{on } \Gamma,\quad \psi\in \cK(\Gamma), \quad  \psi'=0\text{ on }\partial\Gamma,
\end{equation}
where the condition $\psi'|_{\partial\Gamma}=0$ means that   $\psi_1'(a_1)=0,\dots,\psi_n'(a_n)=0$.
This restriction is an eigenfunction of the problem that corresponds to the eigenvalue $\lambda=0$. But the multiplicity of an eigenvalue of \eqref{NeumannPrb} cannot exceed $n-1$ \cite{KacPivovarchik2011}. In addition, the multiplicity $n-1$ is realized for the potential $V(t)=-\pi^2/4$, $t\in \Gamma$.  It remains to note that the restriction operator is an isomorphism from $\Psi_V$ onto the nullspace of \eqref{NeumannPrb}.
\end{proof}

From the proof, it follows that
\begin{equation}\label{LpsiIspsionboundary}
  \ell(\psi)=\psi|_{\partial\Gamma}
\end{equation}
for any $\psi\in\Psi_V$, since $\lim_{|x|\to\infty}\psi_k(x)=\psi_k(a_k)$.
Here and subsequently, we use the same letter for a half-bound state and its restriction to $\Gamma$.

\begin{figure}[t]
\centering
  \includegraphics[scale=0.6]{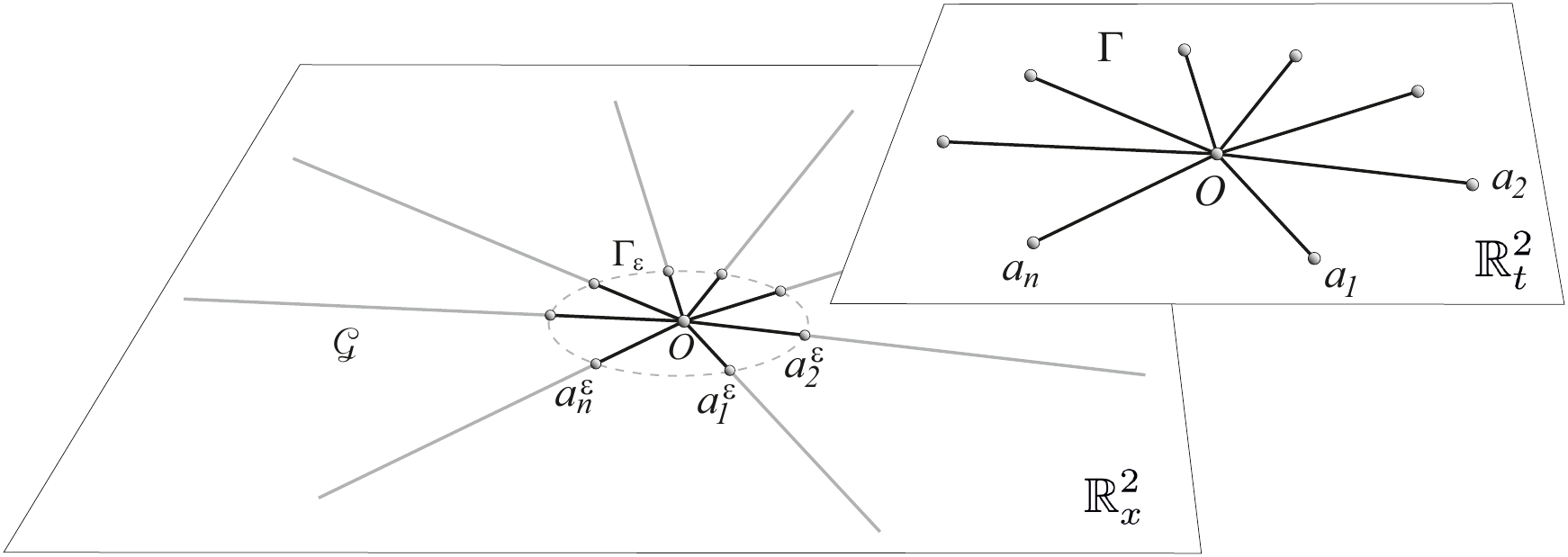}
  \caption{The graph $\Gamma$ is a dilatation of  $\Gamma_\eps$}\label{GammaEps}
\end{figure}

For $f\in L^2(\cG)$ and $\zeta\in \Comp\setminus\Real$, we set $y_\eps=(H_\eps-\zeta)^{-1}f$.
The function $y_\eps$ is a unique solution of  the equation
\begin{equation}\label{ResolventDiffEq}
 -y_\eps''+\big(Q_\eps(x)+\eps^{-2}V(\eps^{-1} x)+\eps^{-1}U(\eps^{-1} x)\big)y_\eps=\zeta y_\eps+f\quad\text{on } \cG
\end{equation}
that belongs to $\dmn H_\eps$.
We look for the  asymptotics of $y_\eps$, as $\eps\to 0$, of the form
\begin{equation}\label{AsymptoticsYepsC1}
 y_\eps(x)\sim
  \begin{cases}
      y(x), & \text{if } x\in \cG_\eps,\\
      u\xe+\eps\ln\eps\, v\xe+ \eps w\xe, & \text{if }  x\in \Gamma_\eps,
 \end{cases}
\end{equation}
where $u$, $v$ and $w$ satisfy the Kirchhoff conditions at the origin.

Since $y_\eps$ is a $C^1$-function on the edges, we glue the approximations on $\partial \Gamma_\eps$ with the conditions
\begin{equation}\label{GlueApprx}
\begin{aligned}
  y(a_k^\eps)&=u(a_k)+\eps\ln\eps \,v(a_k)+ \eps w(a_k)+o(\eps),\\
  y'(a_k^\eps)&=\eps^{-1}u'(a_k)+\ln\eps \,v'(a_k)+ w'(a_k)+o(1),
\end{aligned}
\end{equation}
for all $k=1,\dots,n$. All derivatives are taken in the direction of growth of the parameter $\tau$ in \eqref{Parametrization}.  Here $a_k^\eps=\eps \vec{\imath}_k$ is a boundary vertex of $\Gamma_\eps$ (see Fig.~\ref{GammaEps}).
In view of Proposition~\ref{PropYasymp}, we have $y'(a_k^\eps)=q_k y_k(0)\ln\eps+y_k^{[1]}(0)+o(1)$, as $\eps\to 0$.
Combining this  asymptotics and \eqref{GlueApprx}, we discover $u(a_k)=y_k(0)$, $u'(a_k)=0$, $v'(a_k)=q_ky_k(0)$ and $w'(a_k)=y_k^{[1]}(0)$ for all $a_k\in \partial \Gamma$. Then
\begin{equation}\label{CConds}
u|_{\partial \Gamma}=y(0), \quad u'|_{\partial \Gamma}=0,\quad v'|_{\partial \Gamma}=Ky(0),\quad w'|_{\partial \Gamma}=y^{[1]}(0).
\end{equation}
Since the potential $W_\eps$ is localized on $\Gamma_\eps$ and this graph shrinks to a point as $\eps\to 0$,    $y$ must solve the equation
\begin{equation}\label{AsymptoticsEq}
-y''+Q y=\zeta y+f \quad\text{on } \cG.
\end{equation}
To find coupling conditions for $y$ at $a=0$, we will look closely at equation \eqref{ResolventDiffEq} on $\Gamma_\eps$.  We rewrite it  in terms of the new variables $t=\eps^{-1}x$. Then
\begin{equation}\label{ResolventDiffEqT}
 -u_\eps''+\big(V(t)+\eps\ln\eps\,\kappa(t)+\eps U(t)\big)u_\eps=\eps^2\zeta u_\eps+\eps^2f\quad  \text{on } \Gamma,
\end{equation}
where $u_\eps(t)=y_\eps(\eps t)$. Substituting \eqref{AsymptoticsYepsC1} into \eqref{ResolventDiffEqT} and applying \eqref{CConds} we yield
\begin{align}\label{ProblemU}
  -&u''+Vu=0\quad\text{on }\Gamma,  &&u\in \cK(\Gamma), \quad u'=0\quad\text{on }\partial\Gamma;&\\\label{ProblemV}
  -&v''+Vv=-\kappa u\quad\text{on }\Gamma,  &&v\in \cK(\Gamma), \quad v'=Ky(0)\quad\text{on }\partial\Gamma;&\\\label{ProblemW}
  -&w''+Vw=-U u\quad\text{on }\Gamma,  &&w\in \cK(\Gamma), \quad w'=\qdrz{y}\quad\text{on }\partial\Gamma.&
\end{align}

If $u$ is a non-trivial solution of \eqref{ProblemU}, then it is the restriction of a half-bound state to  $\Gamma$  (cf. the proof of Lemma~\ref{LemNminus1}). We set
\begin{equation}\label{Urepres}
  u(x)=\alpha_1\psi^{(1)}(x)+\dots+\alpha_r\psi^{(r)}(x)\quad\text{for }x\in\Gamma,
\end{equation}
where $\{\psi^{(k)}\}_{k=1}^r$ is a basis of $\Psi_V$. Of course, $u=0$, if $\Psi_V$ is trivial.
Then \eqref{LpsiIspsionboundary} and the first equality in \eqref{CConds} yield
\begin{equation}\label{y0asL}
  y(0)=\alpha_1\psi^{(1)}|_{\partial\Gamma}+\dots+\alpha_r\psi^{(r)}|_{\partial\Gamma}=
  \alpha_1\ell(\psi^{(1)})+\dots+\alpha_r\ell(\psi^{(r)}).
\end{equation}
By \eqref{VectorsLk}, the vectors $l_k=\ell(\psi^{(k)})$ are the columns of matrix $L$. Hence $y(0)=L\alpha$, where
$\alpha=(\alpha_1,\dots,\alpha_r)^T$.
Applying the left inverse $L^+$, we get
\begin{equation}\label{AlphaIsPMy}
\alpha=L^+ y(0).
\end{equation}
Since $\ell(\Psi_V)=\cR_V$, it follows from \eqref{y0asL}  that $y(0)\in\cR_V$. The condition can be written as
\begin{equation}\label{Set1}
  R^* y(0)=0,
\end{equation}
where  $R$ is defined in the proof of Lemma~\ref{SelfadjointnessH}.

\begin{prop}
  Given $g\in L^2(\Gamma)$ and $h\in\Comp$, the problem
  \begin{equation}\label{NHGproblem}
    -\phi''+V\phi=g\quad\text{on }\Gamma, \quad \phi\in \cK(\Gamma), \quad \phi'=h\quad\text{on }\partial\Gamma
  \end{equation}
has a solution if and only if
\begin{equation}\label{CompactibilityCond}
(g,\psi)_{L^2(\Gamma)}+\langle h, \ell(\psi)\rangle_{\Comp^n}=0 \qquad \text{for all \ } \psi\in \Psi_V.
\end{equation}
 If $\Psi_V$ is trivial, then \eqref{NHGproblem} admits a unique solution for all $g$ and $h$.
\end{prop}
\begin{proof}
The statement is a simple consequence of Fredholm’s alternative for
the self-adjoint operator $S\phi=-\phi''+V\phi$ on $\Gamma$, $\dmn S=\{\phi\in\cK(\Gamma)\colon \phi'|_{\partial \Gamma }=0\}$. Condition \eqref{CompactibilityCond} can be easily obtained by multiplying the equation in \eqref{NHGproblem} by  $\bar{\psi}$ and integrating by parts with the use of  the vertex conditions and \eqref{LpsiIspsionboundary}.
\end{proof}

It is enough to check  \eqref{CompactibilityCond} only on the basis vectors:
\begin{equation*}
 (g,\psi^{(k)})_{L^2(\Gamma)}+\langle h, \ell(\psi^{(k)})\rangle_{\Comp^n}=0,\qquad k=1,\dots,r.
\end{equation*}
Applying these solvability conditions to problem \eqref{ProblemV}, we obtain
\begin{equation*}
  -\sum_{j=1}^r \alpha_j(\kappa\psi^{(j)},\psi^{(k)})_{L^2(\Gamma)}+\langle Ky(0), \ell(\psi^{(k)})\rangle_{\Comp^n}=0,\qquad k=1,\dots,r.
\end{equation*}
This collection of equalities can be written as $N\alpha-L^*Ky(0)=0$. Finally, we have
\begin{equation}\label{Set2}
  (NL^+-L^*K)y(0)=0,
\end{equation}
because of \eqref{AlphaIsPMy}.
Similarly, we obtain the solvability condition for  \eqref{ProblemW}:
\begin{equation}\label{Set3}
   ML^+ y(0)-L^*\qdrz y=0.
\end{equation}

We have obtained three sets \eqref{Set1}, \eqref{Set2} and \eqref{Set3} of conditions for the vectors $y(0)$ and $y'(0)$. To correctly define the vertex couplings at $x=0$,  we must select only $n$ linearly independent conditions among them. In view of Lemma~\ref{SelfadjointnessH}, sets \eqref{Set1} and \eqref{Set3} form exactly such a collection, so the vectors $y(0)$ and $y'(0)$ must satisfy them. On the other hand,  the number of linearly independent conditions in \eqref{Set2} is equal to the rank of $NL^+-L^*K$. This rank varies depending on regularizations of the Coulomb potential, i.e. on the function $\kappa$ and the values $q_1,\dots,q_n$. To fulfill \eqref{Set2}, we impose the additional constraint
\begin{equation}\label{ConvergenceCndInProof}
 \ker R^* \subset \ker (NL^+-L^*K).
\end{equation}

Therefore, at least formally, we have that
the leading term $y$ of asymptotics \eqref{AsymptoticsYepsC1} must solve the problem
\begin{equation}\label{ProblemLimitGen}
  -y''+Q y=\zeta y+f \text{ on } \cG, \quad
  y(0)\in\cR_V, \quad ML^+ y(0)-L^*\qdrz y=0,
\end{equation}
provided \eqref{ConvergenceCndInProof} holds. If $\cR_V=\{0\}$, the problem reduces  to the following one
\begin{equation}\label{ProblemDirectSum}
  -y''+Q y=\zeta y+f \text{ on } \cG, \quad y(0)=0.
\end{equation}
Given $f\in L^2(\cG)$ and $\zeta\in \Comp\setminus\Real$, both the problems admit a unique solution.

Then $u$ can be defined by \eqref{Urepres} with $\alpha$ given by \eqref{AlphaIsPMy}. Moreover, both problems \eqref{ProblemV} and \eqref{ProblemW} have solutions defined up to half-bound states. These solutions are uniquely determined by the conditions
\begin{equation}\label{VWOrtho}
  (v,\psi)_{L^2(\Gamma)}=0, \qquad (w,\psi)_{L^2(\Gamma)}=0\quad\text{for all }\psi\in\Psi_V.
\end{equation}

\begin{rem}\rm
Note that condition \eqref{ConvergenceCndInProof} is  necessary for our formal construction. The fact is that any vector $\xi\in \ker R^*$ can be realized as $y(0)$ for some solution $y$ of \eqref{ProblemLimitGen}.  Since $L^*$ is full row rank, there exists a vector $\eta\in \Comp^n$ satisfying the equality $L^*\eta=ML^+ \xi$. Let us consider the function $\phi\in \cW_Q$ such that $\phi(0)=\xi$ and $\qdrz \phi=\eta$. If we set $f=-\phi''+Q\phi-\zeta \phi$, then $y=\phi$ is a solution of \eqref{ProblemLimitGen} and $\xi= y(0)$.  Thus, the condition $(NL^+-L^*K)\xi=0$ must hold for all $\xi\in \cR_V$.
\end{rem}

\subsection{Uniform bounds}\label{Sec32}
We will need some uniform estimates for the terms $y$, $u$, $v$ and $w$ in \eqref{AsymptoticsYepsC1} with respect to the right-hand side $f$.
We have actually derived that $y=(\cH-\zeta)^{-1}f$, where $\cH$ is the operator from Theorem~\ref{MainTheorem}.  Hence the estimate
\begin{equation}\label{EstU}
  \|y\|\leq c_1\|f\|
\end{equation}
holds. From now on,  $\|\,\cdot\,\|$ stands for the norm in $L^2(\cG)$.
The following results are taken from \cite[Lemma~2]{GolovatyJMP2019}, where the problem on the line ($n=2$) has been treated:
\begin{gather}
\label{ULnEst}
    \left\|y|_{\partial \Gamma_\eps}-y(0)\right\|_{\Comp^n}\leq C_1\eps|\ln\eps|\,\|f\|,
    \\
   \left\|y'|_{\partial \Gamma_\eps}-\ln\eps \,Ky(0)-\qdrz y\right\|_{\Comp^n}\leq C_2 \eps^{1/2}\,\|f\|,\\\label{UBEsts}
   \left\|y(0)\right\|_{\Comp^n}+\|\qdrz y \|_{\Comp^n}\leq C_3\|f\|.
\end{gather}
The constants $C_k$ do not depend on $f$ and $\eps$. Proof of  these inequalities is carried out separately on each edge and it does not depend on the number of edges.
\begin{prop}\label{PropUVW}
Let $u$ be the solution of \eqref{ProblemU} given by \eqref{Urepres} and \eqref{AlphaIsPMy} and let  $v$ and $w$  be the solutions of \eqref{ProblemV} and \eqref{ProblemW} satisfying \eqref{VWOrtho}. Then
\begin{equation*}
  \|u\|_{W_2^2(\Gamma)}+\|v\|_{W_2^2(\Gamma)}+\|w\|_{W_2^2(\Gamma)}\leq C\|f\|
\end{equation*}
for all $f\in L^2(\cG)$, where the constant $C$ is independent of $f$.
\end{prop}
\begin{proof}
In view of \eqref{Urepres}, \eqref{AlphaIsPMy} and \eqref{UBEsts}, we deduce
  \begin{align*}
    \|u\|_{W_2^2(\Gamma)}\leq \sum_{k=1}^r |\alpha_k|\|\psi^{(k)}&\|_{W_2^2(\Gamma)}\leq c_1\|\alpha \|_{\Comp^n}\\
    &=c_1\|L^+ y(0)\|_{\Comp^n}\leq c_2\|y(0)\|_{\Comp^n}\leq c_3\|f\|.
  \end{align*}
This bound gives us the uniform estimate for $v$:
\begin{equation*}
  \|v\|_{W_2^2(\Gamma)}\leq c_4(\|\kappa u\|_{L^2(\Gamma)}+\|Ky(0)\|_{\Comp^n})\leq
  c_5(\| u\|_{L^2(\Gamma)}+\|y(0)\|_{\Comp^n})\leq c_6\|f\|.
\end{equation*}
Similarly, repeated application of the bound for $u$ and \eqref{UBEsts} enables us to write
\begin{equation*}
  \|w\|_{W_2^2(\Gamma)}\leq c_7(\|U u\|_{L^2(\Gamma)}+\|\qdrz y\|_{\Comp^n})\leq c_8\|f\|,
\end{equation*}
which completes the proof.
\end{proof}
In order to improve asymptotics \eqref{AsymptoticsYepsC1}, we set
\begin{equation*}
 \hat{y}_\eps(x)=
  \begin{cases}
      y(x), & \text{if } x\in \cG_\eps,\\
      u\xe+\eps\ln\eps \,v\xe+ \eps w\xe+\eps^2 z_\eps\xe, & \text{if }  x\in \Gamma_\eps,
 \end{cases}
\end{equation*}
where $z_\eps$ is a solution of the problem
\begin{equation}\label{CPZeps}
  -z_\eps''+(V-\zeta)z_\eps=f(\eps t)\quad\text{on }\Gamma, \quad z_\eps\in \cK(\Gamma), \quad z_\eps=0\quad\text{on }\partial\Gamma.
\end{equation}

\begin{prop}\label{PropZeps}
 For all $f\in L^2(\cG)$ and $\zeta\in \Comp\setminus\Real$, the function $z_\eps$ admits the estimate
\begin{equation}\label{ZepsEst}
\|z_\eps\|_{W_2^2(\Gamma)}\leq C\eps^{-1/2}\|f\|
\end{equation}
with $C$ independent of $f$ and $\eps$.
\end{prop}
\begin{proof}
First we see that
\begin{equation*}
\|f(\eps\,\cdot)\|_{L^2(\Gamma)}^2=\int_{\Gamma}|f(\eps t)|^2\,d\Gamma
=\eps^{-1} \int_{\Gamma_\eps}|f(x)|^2\,d\cG\leq c_1\eps^{-1}\|f\|^2
\end{equation*}
for all $f\in L^2(\cG)$.  Here $d\Gamma$ is the Lebesgue induced measure on $\Gamma\subset\Real^2_t$. Of course, $d\Gamma=\eps^{-1} d\cG$, because $\Gamma$ is the dilatation of  $\Gamma_\eps$ by the factor $\eps^{-1}$.
Since $\Im \zeta\neq 0$, there exists a unique solution of \eqref{CPZeps} such that
\begin{equation*}
   \|z_\eps\|_{W_2^2(\Gamma)}\leq c_2\|f(\eps\,\cdot)\|_{L^2(\Gamma)},
 \end{equation*}
and \eqref{ZepsEst} is proved.
\end{proof}

\begin{figure}[h]
\centering
  \includegraphics[scale=1.1]{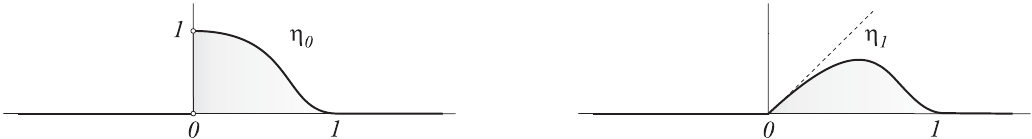}
  \caption{The plots of $\eta_0$ and $\eta_1$}\label{Jumps}
\end{figure}

In general, $\hat{y}_\eps$ has jump discontinuities on $\partial \Gamma_\eps$ and therefore it does not belong to
$\dmn H_\eps$.  But there exists a corrector $\rho^\eps$ with a small $L^2$-norm, as $\eps\to 0$, such that
$Y_\eps:=\hat{y}_\eps+\rho^\eps\in \dmn H_\eps$.
We will define $\rho^\eps$ as follows.
Suppose $\eta_0,\eta_1\colon \Real\to\Real$ are the functions that are smooth outside the origin and  have compact supports contained in $[0,1]$. In addition,
$\eta_0(+0)=1$, $\eta_0'(+0)=0$, $\eta_1(+0)=0$ and $\eta_1'(+0)=1$ (see Fig.~\ref{Jumps}).
Given an edge $e_k$, we set
\begin{equation}\label{CorrectoR}
\rho^\eps_k(\tau)=-[\hat{y}_\eps]_{x=a_k^\eps}\,\eta_0(\tau-\eps)-[\hat{y}_\eps']_{x=a_k^\eps}\,\eta_1(\tau-\eps)
\end{equation}
for $\tau\geq 0$, where $[h]_{x=b}$ is a jump of a function $h$ at the point $b$.
It is easy to check that $[\rho^\eps_k]_{\tau=\eps}=-[\hat{y}_\eps]_{x=a_k^\eps}$ and
$\big[\frac{d\rho^\eps_k}{d\tau }\big]_{\tau=\eps}=-[\hat{y}_\eps']_{x=a_k^\eps}$. From this we conclude that  $Y_\eps\in W_2^{2,loc}(\cG\setminus\{0\})$. Also, we see that $\rho^\eps=0$ on $\Gamma_\eps$, so the function $Y_\eps$ satisfies Kirchhoff's condition at the origin. Hence, $Y_\eps\in\dmn H_\eps$.

\begin{prop}\label{LemmaRho}
 Assume that  $\rho^\eps$ is  given by \eqref{CorrectoR}. There exist constants $C_1$ and $C_2$ independent of $f$ such that
  \begin{gather}\label{supRho}
    \|\rho^\eps\|_{C^2(\cG_\eps)}\leq C_1\eps^{1/2}\|f\|,\\\nonumber
   \|Q\rho^\eps\|\leq C_2\eps^{1/4}\|f\|.
  \end{gather}
\end{prop}
\begin{proof}
Let $[\hat{y}_\eps]_{\partial \Gamma_\eps}$ and $[\hat{y}_\eps']_{\partial \Gamma_\eps}$ be the vectors in $\Comp^n$ with  entries $[\hat{y}_\eps]_{x=a_k^\eps}$ and $[\hat{y}_\eps']_{x=a_k^\eps}$, $k=1,\dots,n$, respectively.
To prove \eqref{supRho} it suffices to show
\begin{equation*}
  \big\|[\hat{y}_\eps]_{\partial \Gamma_\eps}\big\|_{\Comp^n}+ \big\|[\hat{y}_\eps']_{\partial \Gamma_\eps}\big\|_{\Comp^n}\leq c\eps^{1/2}\|f\|,
\end{equation*}
since the functions $\eta_j$ in \eqref{CorrectoR} are smooth and bounded on $\cG_\eps$ together with all their derivatives.
Combining \eqref{ULnEst}-\eqref{UBEsts}, Propositions~\ref{PropUVW} and \ref{PropZeps}, and the continuity  of embedding
$W_2^2(\Gamma)\subset C^1(\Gamma)$, we conclude that
\begin{align*}
  &\begin{aligned}
    \big\|[\hat{y}_\eps]_{\partial \Gamma_\eps}\big\|_{\Comp^n}&=\big\|y|_{\partial \Gamma_\eps}-u|_{\partial \Gamma}-\eps\ln\eps \:v|_{\partial \Gamma}-\eps w|_{\partial \Gamma}\big\|_{\Comp^n}\\
    &\leq \big\|y|_{\partial \Gamma_\eps}-y(0)\big\|_{\Comp^n}+\eps|\ln\eps| \big\|v|_{\partial \Gamma}\big\|_{\Comp^n}+\eps \big\|w|_{\partial \Gamma}\big\|_{\Comp^n}\leq  c_1 \eps|\ln\eps| \,\|f\|,
  \end{aligned}
  \\
  &\begin{aligned}
    \big\|[\hat{y}_\eps']_{\partial \Gamma_\eps}\big\|_{\Comp^n}&=\big\|y'|_{\partial \Gamma_\eps}-\eps^{-1}u'|_{\partial \Gamma}-\ln\eps \:v'|_{\partial \Gamma}- w'|_{\partial \Gamma}-\eps z_\eps'|_{\partial \Gamma}\big\|_{\Comp^n}\\
    &\leq \big\|y'|_{\partial \Gamma_\eps}-\ln\eps \:Ky(0)-\qdrz y\big\|_{\Comp^n}+\eps \big\|z_\eps'|_{\partial \Gamma}\big\|_{\Comp^n} \leq  c_2 \eps^{1/2} \,\|f\|.
  \end{aligned}
\end{align*}
Set $\cG(\eps_1, \eps_2)=\cG\cap\{x\colon \eps_1\leq|x|\leq\eps_2\}$.
Since $|\eta_1(\tau)|\leq c\tau$ as $\tau\to +0$, we have
\begin{multline}\label{EstRhoNear0}
  \sup_{\cG(\eps, \eps^{1/4})}|\rho^\eps|\leq c_3
  \big(\big\|[\hat{y}_\eps]_{\partial \Gamma_\eps}\big\|_{\Comp^n}+ \big\|[\hat{y}_\eps']_{\partial \Gamma_\eps}\big\|_{\Comp^n}\eps^{1/4}\big)
  \\
\leq c_4(\eps|\ln\eps| + \eps^{3/4})\|f\|\leq c_5\eps^{3/4}\|f\|.
\end{multline}
Recall that $\rho^\eps=0$ outside of $\cG(\eps, 2)$, provided  $\eps$ is small. Then utilizing estimates \eqref{supRho} and \eqref{EstRhoNear0},  we obtain the bound
\begin{align*}\allowdisplaybreaks
  \|Q\rho^\eps\|^2&= \int\limits_{\cG(\eps, 2)}\kern-2pt Q^2|\rho^\eps|^2\,d\cG
  \leq
  \max\limits_{1\leq k\leq n} |q_k|^2\kern-2pt\int\limits_{\cG(\eps, 2)}\kern-2pt |x|^{-2}|\rho^\eps|^2\,d\cG
  \\\allowdisplaybreaks
  &\leq c_6\left(\;\int\limits_{\cG(\eps, \eps^{1/4})}\kern-2pt |x|^{-2}|\rho^\eps|^2\,d\cG
  +\kern-2pt\int\limits_{\cG(\eps^{1/4}, 2)}\kern-2pt |x|^{-2}|\rho^\eps|^2\,d\cG\right)
  \\\allowdisplaybreaks
   &\leq c_6\sup_{\cG(\eps, \eps^{1/4})}|\rho^\eps|^2
   \kern-2pt\int\limits_{\cG(\eps, \eps^{1/4})}\kern-2pt |x|^{-2}\,d\cG
  +c_6\sup_{\cG(\eps^{1/4}, 2)} |x|^{-2}|\rho^\eps|^2
 \kern-2pt \int\limits_{\cG(\eps^{1/4}, 2)}\kern-4pt d\cG
  \\
  &\leq c_7 \left(\eps^{3/2} \kern-2pt \int\limits_{\cG(\eps, \eps^{1/4})}\kern-2pt |x|^{-2}\,d\cG+\eps^{1/2}\right)\|f\|^2
  \leq c_8\eps^{1/2}\|f\|^2,
\end{align*}
which completes the proof.
\end{proof}

\subsection{End of the proof}
We will now prove that the remainder term
\begin{equation}\label{ResEqnWithRemnd}
  g_\eps=(H_\eps-\zeta)Y_\eps-f
\end{equation}
is small in the $L_2$-norm uniformly with respect to $f$. For $x\in\cG_\eps$,  we have
\begin{multline*}
  g_\eps(x)=-\big(y(x)+\rho^\eps(x)\big)''+\big(Q(x)-\zeta\big)\big(y(x)
  +\rho^\eps(x)\big)-f(x)\\=-(\rho^\eps)''(x) +(Q(x)-\zeta)\rho_\eps(x),
\end{multline*}
by \eqref{AsymptoticsEq}. If $x\in\Gamma_\eps$, then
\begin{align*}
   g_\eps(x)  &= \eps^{-2} \Big(-u''\xe+V\xe u\xe\Big)
    \\
       &+\eps^{-1}\ln\eps \Big(-v''\xe+V\xe v\xe+\kappa\xe u\xe\Big)
    \\
     &  +\eps^{-1}\Big(-w''\xe+ V\xe w+U\xe u\xe\Big)
    \\
     &+\Big(-z''_\eps\xe+(V\xe-\zeta) z_\eps\xe-f(x)\Big)
     \\
     &+\ln\eps \:\kappa\xe\big(\ln\eps \,v\xe+ w\xe+\eps z_\eps\xe\big)\\
     &+U\xe\big(\ln\eps \,v\xe+w\xe+\eps z_\eps\xe\big)
     - \zeta\big(u\xe+\ln\eps \,v\xe+\eps w\xe\big)
     \\
    &=\big( \ln\eps\,\kappa\xe+U\xe\big)\big(\ln\eps \,v\xe+w\xe+ \eps z_\eps\xe\big)\\
    &\hskip120pt-\zeta\big(u\xe+\ln\eps \,v\xe+\eps w\xe\big)
\end{align*}
by \eqref{ProblemU}--\eqref{ProblemW} and \eqref{CPZeps}. Hence we have
\begin{multline*}
 \|g_\eps\|\leq \|(\rho^\eps)''+\zeta\rho^\eps\|+\|Q\rho^\eps\|\\
 +\sup_{t\in \cG}\big(|\ln\eps||\kappa(t)|+|U(t)|\big)
 \|\ln\eps \,v\xep+w\xep+\eps z_\eps\xep\|_{L^2(\Gamma_\eps)}\\
 +|\zeta|\,\|u\xep+\ln\eps \,v\xep+w\xep\|_{L^2(\Gamma_\eps)}
 \leq c_1(\eps^{1/2}+\eps^{1/4})\|f\|\\
 +c_2 \eps^{1/2} |\ln\eps|\, \|\ln\eps\, v+w+\eps z_\eps\|_{L^2(\Gamma)}
 +|\zeta|\eps^{1/2}\,\|u+v\ln\eps+w\|_{L^2(\Gamma)}
 \leq c\eps^{1/4}\|f\|
\end{multline*}
in view of Propositions~\ref{PropUVW}-\ref{LemmaRho}.  Here we have also used the equality
\begin{equation*}
  \|h\xep\|_{L^2(\Gamma_\eps)}=\eps^{1/2}\|h\|_{L^2(\Gamma)}
\end{equation*}
for each $h\in L^2(\Gamma)$. We conclude from \eqref{ResEqnWithRemnd} that
\begin{equation*}
  Y_\eps=(H_\eps-\zeta)^{-1}(f+ g_\eps)=y_\eps+(H_\eps-\zeta)^{-1}g_\eps,
\end{equation*}
hence that
\begin{equation}\label{estYeps-Ueps}
   \|y_\eps-Y_\eps\|=\|(H_\eps-\zeta)^{-1}g_\eps\|\leq |\Im\zeta|^{-1}\|g_\eps\|\leq c|\Im\zeta|^{-1}\eps^{1/4}\|f\|.
\end{equation}
Now let us consider the difference
\begin{equation*}
 Y_\eps(x)-y(x)=
  \begin{cases}
      \rho^\eps(x), & \text{if } x\in \cG_\eps,\\
      u\xe+v\xe\eps\ln\eps+ w\xe\eps+z_\eps\xe \eps^2-y(x), & \text{if }  x\in \Gamma_\eps,
 \end{cases}
\end{equation*}
We can as before invoke Propositions~\ref{PropUVW}-\ref{LemmaRho} and bound \eqref{EstU} to derive
\begin{multline}\label{estUeps-U}
   \|Y_\eps-y\|\leq \|\rho^\eps\|+\eps^{1/2}\| u+\eps\ln\eps \,v+ \eps w+\eps^2 z_\eps\|_{L^2(\Gamma)}
   \\+\|y\|_{L_2(\Gamma_\eps)}
   \leq c_1\eps^{1/2}(\|f\|+\max_{\Gamma_\eps}|y|)\leq c_2\eps^{1/2}\|f\|.
\end{multline}
Recalling the definitions of $y_\eps$ and $y$, we estimate
\begin{multline*}
    \|(H_\eps-\zeta)^{-1}f-(\mathcal{H}-\zeta)^{-1}f\|=\|y_\eps-y\|\\
    \leq\|y_\eps-Y_\eps\|+ \|Y_\eps-y\|
      \leq C (1+|\Im \zeta|^{-1})\eps^{1/4}\|f\|,
\end{multline*}
by \eqref{estYeps-Ueps} and \eqref{estUeps-U}. The last bound establishes
the norm resolvent convergence of  $H_\eps$ to the operator $\mathcal{H}$ and the estimate in Theorem~\ref{TheoremEst}.

\section{Proof of Theorem~\ref{OprQConvergenceToDirectSum}}\label{Sect4}
We now suppose that $W_\eps(x)=Q_\eps(x)+\eps^{-1}U(\eps^{-1}x)$ and the regularization $Q_\eps$ of a Coulomb-type potential $Q$ is given by \eqref{Qeps}. Assume $Q_\eps$ does not satisfy condition \eqref{QPlusDeltaConvCond}, but $H_\eps$ converge to an operator $H_0$ in the strong resolvent topology.

Given $\zeta\in\Comp\setminus\Real$ and $f\in L^2(\cG)$ , we set $y_\eps=(H_\eps-\zeta)^{-1}f$ and $y=(H_0-\zeta)^{-1}f$. The function $y_\eps$ is a solution of the equation $-y_\eps''+Qy_\eps=\zeta y_\eps+f$ on $\cG\setminus\Gamma_\eps$. Then the strong resolvent convergence implies that $y_\eps\to y$ in $L^2(\cG)$, hence that $y$ solves the equation
$-y''+Q y=\zeta y+f$ on $\cG$.

\begin{prop}\label{PropYepsY0}
As $\eps\to 0$, we have
\begin{itemize}
    \item[\textit{(i)}] $y_\eps|_{\partial \Gamma_\eps}\to y(0)$;
    \item[\textit{(ii)}] $\max\limits_{x\in\Gamma_\eps}|y_\eps(x)|\leq c$ with constant being independent of $\eps$.
\end{itemize}
\end{prop}
\begin{proof}
\textit{(i)} Fix an edge $e_k$. Then $y_{\eps,k}$ is a family of solutions of the equation
\begin{equation}\label{EqnForYeps}
-\frac{d^2Y}{d\tau^2}+\frac{q_k}{\tau}\,Y=\zeta Y+f(\tau)
\end{equation}
for $\tau\in [\eps,+\infty)$. This family admits the representation
\begin{equation}\label{SimpleRepes}
    y_{\eps,k}(\tau)=b_\eps^k Y_{\zeta,0}(\tau)+Y_{\zeta,f}(\tau),\quad \tau\geq \eps,
\end{equation}
where $b_\eps^k$ is a constant,  $Y_{\zeta,f}$ is a $L^2(\Real_+)$-solution of \eqref{EqnForYeps} and $Y_{\zeta,0}$ is a $L^2(\Real_+)$-solution of the corresponding homogeneous equation. The solution $Y_{\zeta,0}$ can be expressed in terms of the Whittaker functions (see \cite{Moshinsky1993} for details). In view of Proposition~\ref{PropYasymp}, both $Y_{\zeta,0}$ and $Y_{\zeta,f}$ are bounded in a neighbourhood of the origin and have finite values $Y_{\zeta,0}(+0)$ and $Y_{\zeta,f}(+0)$. Moreover $b_\eps^k\to b_0^k$, as $\eps\to 0$, because of $y_{\eps,k}\to y_k$ in $L^2(e_k)$. Hence,
\begin{equation*}
  y_{\eps,k}(a^\eps_k)\to b_0^k Y_{\zeta,0}(+0)+Y_{\zeta,f}(+0),\quad \text{as } \eps\to 0,
\end{equation*}
for each $k=1,\dots,n$. It remains to notice that $y_k=b_0^k Y_{\zeta,0}+Y_{\zeta,f}$ on $e_k$.

\textit{(ii)} The function $u_\eps(t)=y_\eps(\eps t)$ is a solution of the boundary value problem
\begin{gather*}
  u_\eps''=\eps\ln\eps\; \kappa(t) u_\eps+\eps U(t) u_\eps-\eps^2\zeta u_\eps-\eps^2f(\eps t),\quad t\in\Gamma,\\
  u_\eps\in\cK(\Gamma), \quad   u_\eps|_{\partial \Gamma}=y_\eps|_{\partial \Gamma_\eps}.
\end{gather*}
Then the standard elliptic estimate gives us
\begin{equation*}
\|u_\eps\|_{W_2^2(\Gamma)}\leq c_1\big(\|y_\eps|_{\partial \Gamma_\eps}\|_{\Comp^n}+\eps|\ln\eps|\;\|u_\eps\|_{L^2(\Gamma)}
+\eps^2\|f(\eps\,\cdot)\|_{L^2(\Gamma)}\big)\leq c_2.
\end{equation*}
Indeed, the norm $\|y_\eps|_{\partial \Gamma_\eps}\|_{\Comp^n}$ is bounded as $\eps\to 0$, by part \textit{(i)}. Moreover
\begin{equation*}
 \|u_\eps\|_{L^2(\Gamma)}\leq c_3\eps^{-1/2}\|y_\eps\|, \qquad \|f(\eps\,\cdot)\|_{L^2(\Gamma)}\leq c_4\eps^{-1/2}\|f\|.
\end{equation*}
By the Sobolev imbedding theorems, we have $\max\limits_{t\in\Gamma}|y_\eps(\eps t)|\leq c$, which is the desired conclusion.
\end{proof}

\begin{prop}\label{PropYisCont}
 The function $y$ is continuous at the vertex $a=0$, i.e.,
 \begin{equation*}
   y_1(0)=y_2(0)=\cdots=y_n(0).
 \end{equation*}
\end{prop}
\begin{proof}
Given $i$ and $j$, we introduce the function
\begin{equation*}
  h_{ij}(x)=
  \begin{cases}
    \phantom{-}|x|& \text{for } x\in e_i,\\
    -|x|& \text{for } x\in e_j,\\
    \phantom{-1}0& \text{on other edges}
  \end{cases}
\end{equation*}
that satisfies the Kirchhoff conditions at the origin.
Multiplying the equation
\begin{equation}\label{EqnOnGammaEps}
  -y_\eps''+\eps^{-1}\ln\eps\; \kappa\xe y_\eps+\eps^{-1}U\xe y_\eps=\zeta y_\eps+f(x), \quad x\in \Gamma_\eps
\end{equation}
by  $h_{ij}$ and integrating by parts yield
 \begin{multline}\label{Ch6DifferenceY}
  y_\eps(a^\eps_i)-y_\eps(a^\eps_j)=\eps \big(y_\eps'(a^\eps_i)-y_\eps'(a^\eps_j)\big)\\-
  \int_{\Gamma_\eps}\left(\eps^{-1}\ln\eps\,\kappa\xe+\eps^{-1}U\xe-\zeta- f(x)\right)h_{ij}(x)y_\eps(x)\,d\cG.
\end{multline}
By \eqref{LogAsympOfDerv} and \eqref{SimpleRepes}, we have $|y_\eps'(a^\eps_k)|\leq c_1|\ln \eps|$, as $\eps\to 0$, for all $k=1,\dots,n$.
Set $e_k^\eps=e_k\cap\{x\colon |x|\leq \eps\}$. Then the estimate
\begin{equation*}
  \left|\int_{\Gamma_\eps}h_{ij}(x)w_\eps(x)\,d\cG\right|\leq\left|\int_{e_i^\eps\cup e_j^\eps}|x||w_\eps(x)|\,d\cG\right|\leq c\eps^{3/2}
\end{equation*}
holds, provided the sequence $\|w_\eps\|_{L^2(\Gamma_\eps)}$ is bounded as $\eps\to 0$. In view of Proposition~\ref{PropYepsY0}\textit{(ii)}, the right hand side of \eqref{Ch6DifferenceY} tends to zero as $\eps\to 0$.
Hence
\begin{equation*}
 y_\eps(a^\eps_i)-y_\eps(a^\eps_j)\to 0,
\end{equation*}
 so $y_i(0)=y_j(0)$, by part \textit{(i)} of Proposition~\ref{PropYepsY0}.
\end{proof}

Integrating equation \eqref{EqnOnGammaEps} over $\Gamma_\eps$, we obtain
\begin{multline*}
  \sum_{k=1}^n y_\eps'(a^\eps_k)=\eps^{-1}\ln\eps\int_{\Gamma_\eps}\kappa\xe y_\eps(x)\,d\cG\\
  +\eps^{-1}\int_{\Gamma_\eps}U\xe y_\eps(x)\,d\cG-
  \int_{\Gamma_\eps}\left(\zeta y_\eps(x)+f(x)\right)\,d\cG.
\end{multline*}
With the asymptotics $y_\eps'(a^\eps_k)=q_k y_\eps(a^\eps_k)\ln\eps+O(1)$ as $\eps\to 0$, the last equality can be rewritten as
\begin{equation}\label{TheLastEql}
  \sum_{k=1}^n q_ky_\eps(a^\eps_k)-\eps^{-1}\int_{\Gamma_\eps}\kappa\xe y_\eps(x)\,d\cG=O(|\ln\eps|^{-1}).
\end{equation}
In view of Propositon~\ref{PropYisCont}, we have
\begin{equation*}
  \eps^{-1}\int_{\Gamma_\eps}\kappa\xe y_\eps(x)\,d\cG\to \sum_{k=1}^n y_k(0)\int_{e_k}\kappa(t)\,dt= y_1(0)\int_{\Gamma}\kappa\,d\Gamma.
\end{equation*}
Passing to the limit as $\eps\to 0$ in \eqref{TheLastEql}, we find
\begin{equation*}
  \left(\sum_{k=1}^n q_k-\int_{\Gamma}\kappa\,d\Gamma\right)y_1(0)=0.
\end{equation*}
Hence, if condition \eqref{QPlusDeltaConvCond} is not fulfilled, then $y_1(0)=0$, and so $y(0)$ is the zero vector, by Proposition~\ref{PropYisCont}. Therefore we have proved that the limit function $y$ solves problem \eqref{ProblemDirectSum}, and finally that the operator $H_0$, if it exists, is the direct sum  $\cD_{e_1}\oplus\cdots\oplus\cD_{e_n}$.

\end{document}